\documentclass[10pt,twoside]{siamart1116}
\usepackage{latexsym,amssymb, amsmath,mathrsfs, mathtools,tikz}
\usepackage[OT2,T1]{fontenc}

\usepackage{setspace}
\usepackage{graphicx, subfigure, multicol}
\usetikzlibrary{positioning}

\newsiamthm{conjecture}{Conjecture} 
\newsiamremark{remark}{Remark}
\newcommand{\McC}{\raise.5ex\hbox{c}}

\newcommand{\IM}{\mathrm{Im}\,}
\numberwithin{theorem}{section}

\title{Stable Polynomials via Undirected Colored Graphs%
  \thanks{Submitted June 14, 2025
  \date.
\funding{Both authors were partially supported by National Science Foundation
DMS grant \#2000088}}}

\author{Kelly Bickel%
  \thanks{Department of Mathematics, Bucknell University, 360 Olin Science Building, Lewisburg, PA 17837, USA (\email{kelly.bickel@bucknell.edu}).}%
  \and
  Yang Hong%
  \thanks{Department of Mathematics, University of Maryland, College Park, 4326 Atlantic Building, College Park, MA 20742, USA
    (\email{yh013@bucknell.edu})}
}

\begin{document}
\maketitle

\begin{abstract} This paper initiates a systematic study of connections between undirected colored graphs and associated two-variable stable polynomials obtained via Cauchy transform-type formulas. Examples of such stable polynomials have played crucial roles in other recent studies, though their general properties have remained rather opaque. Using linear algebra techniques, this paper characterizes when these stable polynomials have boundary zeros and studies the finer behavior, called contact order, of their zero sets near the guaranteed boundary zeros.
\end{abstract}

 \begin{keywords}Stable polynomials, colored graphs, rational inner functions\end{keywords}
  \begin{AMS} Primary: 32A08,
Secondary: 
05C50, 32A40, 47A48 \end{AMS}

\section{Introduction}
\subsection{Key Idea}
\raggedbottom

A two-variable polynomial $p$ is said to be stable on a domain $\Omega$ if it has no zeros on $\Omega$. In this paper, we will generally let $\Omega$ be the bidisk $\mathbb{D}^2$, where  $\mathbb{D} = \{ z\in \mathbb{C} : |z| <1 \}$.
For example, $q(z_1, z_2) = 2-z_1-z_2$ is stable on $\mathbb{D}^2$, though stable polynomials can be much more complicated. Their structure and properties are still being explored, see \cite{bps18, bps19a, bpsk22, Kne10, Kne15, Sola23}, and they (and related objects) have appeared recently as key tools in a number of fields including probability theory, control
theory, electrical engineering, combinatorics, graph theory, physics, and dynamical systems, see \cite{BSV05, BBL09, lsv13,  Kum02, KS20, MSS1, MSS2,  W11}.

In this paper, motivated by work of Pascoe in \cite{Pascoe18}, we initiate a study of connections between stable polynomials and undirected, colored graphs. Specifically, we will start with a  graph and use its adjacency matrix to construct a (generally quite nontrivial) two-variable stable polynomial; this construction method had yielded crucial examples in recent papers \cite{bps19a, bps19b, bpsk22}, though the properties of general polynomials obtained via this construction have not been previously studied. 

We illustrate this construction with an example: starting with the $4$-vertex colored graph given in Figure \ref{fig:4vertexgraph}, we can extract its adjacency matrix $A$ and a matrix $Y$ representing the coloring of blue on vertices $1-3$ and of red  on vertex $4$ (a blue coloring is also indicated by circle-shaped vertices and a red coloring is indicated by square-shaped vertices).
\begin{figure}[ht!] 
\begin{minipage}{0.30\linewidth}
\begin{center}
\begin{tikzpicture}[
roundnode/.style={circle, draw=blue!60, fill=blue!5, very thick, minimum size=7mm},
squarednode/.style={rectangle, draw=red!60, fill=red!5, very thick, minimum size=7mm}, 
]
\node[roundnode]      (node1)                              {1};
\node[roundnode]        (node2)       [below=of node1] {2};
  \node[roundnode]    (node3)       [right=of node2] {3};
 \node[squarednode]       (node4)       [right=of node1] {4};
\draw[-] (node1.south) -- (node2.north);
\draw[-] (node2.east) -- (node3.west);
\draw[-]  (node4.west) -- (node1.east);
\draw[-]  (node4.south) -- (node3.north);
\end{tikzpicture}
\end{center}
\end{minipage}%
\begin{minipage}{0.7\linewidth}
\[ A = \begin{bmatrix} 0 & 1 & 0 &1 \\ 
 1 & 0 & 1 &0 \\
  0 & 1 & 0 &1 \\
   1 & 0 & 1 &0 \end{bmatrix}, \  \ \ Y = \begin{bmatrix} 1 & 0 & 0 &0 \\ 
 0 & 1 & 0 &0 \\
  0 & 0 & 1 &0 \\
   0 & 0 & 0 &0 \end{bmatrix} \hspace{.4in} \]
\end{minipage}
\caption{A graph with adjacency matrix A and coloring matrix $Y$.}
\label{fig:4vertexgraph}
\end{figure}
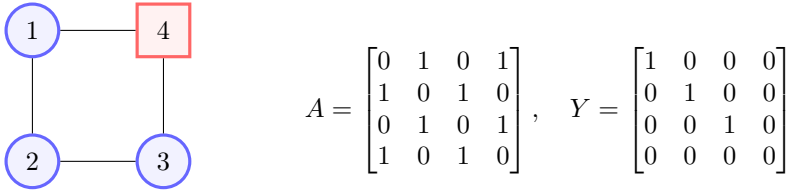
Define the associated function
\begin{equation} \label{eqn:exf}f(z_1, z_2):= \Big( A - z_1Y - z_2( I - Y) \Big)^{-1}_{11} = \frac{z_1 + z_2 -z_1^2z_2}{z_1^3z_2-2z_1^2-2z_1z_2}, \end{equation}
where we took the $(1,1)$ entry of the inverse of the matrix $A - z_1Y - z_2( I - Y)$. Because $A$ is self-adjoint and $0 \le Y\le I$ (where $0$ and $I$ are the zero and identity matrices respectively), the function $f$
maps $\mathbb{H}^2$ to $\mathbb{H}$, where $\mathbb{H} = \{ z \in \mathbb{C}: \IM(z) >0 \}$ is the upper-half plane. After composing $f$ with conformal maps between $\mathbb{H}$ and $\mathbb{D}$ and extracting the denominator, one can obtain 
\begin{equation} \label{eqn:firstp} p(z_1, z_2) = 4-z_2-z_1z_2-3z_1^2z_2+z_1^3z_2,\end{equation}
which is a stable polynomial on $\mathbb{D}^2.$
In general, this method yields exotic stable polynomials whose properties reflect the properties of the original graph. This construction and associated investigations fit into a larger body of ongoing research on connections between graph theory, linear algebra, and complex analysis, see for example \cite{ATTT23, AHL22, AT15, H23}.  

\subsection{Background and General Setup} The construction above is rooted in representation theory for two-variable holomorphic functions. In one variable, if $f:\mathbb{H} \rightarrow \overline{\mathbb{H}}$ is holomorphic and $\liminf_{y\rightarrow \infty} |y f(iy)| <\infty$, then Nevanlinna's representation theorem says that $f$ is the Cauchy transform of a finite positive Borel measure $\mu$ on $\mathbb{R}$, namely:
\[ f(z) = \int_{\mathbb{R}} \frac{d\mu(t)}{t-z} = \left \langle (A - zI)^{-1} 1, 1 \right \rangle_{L^2(\mu)}, \]
where $A$ is multiplication by the independent variable on $L^2(\mu)$ and $1$ is the constant function.

In the two-variable setting, Agler, Tully-Doyle, and Young in \cite[Theorem 1.6]{ATDY16} obtained a useful generalization of Nevanlinna's representation theorem. 
They proved that for any holomorphic $f: \mathbb{H}^2 \rightarrow \overline{\mathbb{H}}$ with $\liminf_{y\rightarrow \infty} |f(iy,iy) y | <\infty$, there is a Hilbert space $\mathcal{H}$, a densely-defined, self-adjoint operator $A$ on $\mathcal{H}$,  a positive operator $Y$ on $\mathcal{H}$ with $0 \le Y \le I$, and a vector $\alpha \in \mathcal{H}$ so that 
\[ f(z_1, z_2) = \left \langle (A - z_1Y - z_2(I-Y))^{-1} \alpha, \alpha \right \rangle_{\mathcal{H}},\]
for $z\in \mathbb{H}^2$.  A careful reading of Agler's earlier work \cite{Agler90} implies that if $f$ is rational and real almost everywhere on $\mathbb{R}^2$, then we can take the Hilbert space $\mathcal{H}$ to be finite-dimensional, so that $A$ can be viewed as an $n \times n$ self-adjoint matrix and $Y$ as a positive semidefinite matrix. Then, we can write 
\begin{equation} \label{eqn:RRP} f(z_1, z_2) = \alpha^\star (A - z_Y)^{-1} \alpha,\end{equation}
where $\alpha^\star$ is the conjugate transpose of $\alpha$ and $z_Y$  is shorthand for $z_1 Y + z_2(I-Y).$ This is exactly the form of $f$ in \eqref{eqn:exf} with $\alpha = e_1,$ the first standard basis vector of $\mathbb{R}^n.$ To move to $\mathbb{D}^2$, one can use the following conformal maps from $\mathbb{D}$ to $\mathbb{H}$ and back:
\begin{equation} \label{eqn:beta}
\beta:\mathbb{H} \rightarrow \mathbb{D}, \text{ with }  \beta(z)=  \frac{1 + zi}{1 - zi} \ \text{ and } \ 
\beta^{-1}: \mathbb{D}\rightarrow \mathbb{H},  \text{ with }    \beta^{-1}(z) = i\left(\frac{1-z}{1+z}\right).
\end{equation}

If $f$ is defined as  in \eqref{eqn:RRP} with $\alpha \in \mathbb{R}^n$, then $\phi: = \beta \circ f \circ \beta^{-1}$ is a rational function holomorphic on $\mathbb{D}^2$ with $|\phi|=1$ almost everywhere on $\mathbb{T}^2$, where $\mathbb{T} = \{ z\in \mathbb{C} : |z| =1\}$ is the unit circle.  Such functions $\phi$ are called \emph{rational inner functions} and have denominators which are stable polynomials on $\mathbb{D}^2$. Moreover, given a stable polynomial $p$ on $\mathbb{D}^2$ with bidegree $(m,n)$, its reflection polynomial $\tilde{p}$ is defined to be
\begin{equation} \label{eqn:tildep} \tilde{p}(z_1, z_2) = z_1^m z_2^n \overline{p\left( \frac{1}{\bar{z}_1},  \frac{1}{\bar{z}_2}\right)} \end{equation}
and any rational inner function $\phi$ with denominator $p$ is of the form
\begin{equation} \label{eqn:RIFform} \phi(z_1, z_2) = \lambda z_1^k z_2^\ell \frac{\tilde{p}(z_1, z_2)}{p(z_1,z_2)}, \end{equation}
where $\lambda \in \mathbb{T}$
 and $k, \ell \in \mathbb{N}$, see \cite{Rud69}. If $\phi$ is written so that its numerator and denominator have no common factors, then its stable polynomial denominator will also have no zeros on the faces of the bidisk $(\mathbb{D} \times \mathbb{T}) \cup (\mathbb{T} \times \mathbb{D})$, see Lemma 10.1 in \cite{Kne15}.
Throughout this paper, when we refer to the denominator of a rational inner function, we will assume that any common factors between numerator and denominator have already been canceled.

Rational inner functions on  $\mathbb{D}^2$ are the two-variable analogues of important one-variable functions called finite Blaschke products and play crucial roles in two-variable interpolation, approximation, and matrix monotonicity questions, see \cite{AglMcCbook, AMY12b, Rud69}. The behavior of rational inner functions near boundary singularities (the boundary zeros of their stable polynomial denominators) can be complicated and a number of properties (including derivative integrability, non-tangential polynomial approximations, and unimodular level set structures) have been studied in \cite{bps18, bps19a, bpsk22}. 

In this paper, we construct functions $f$ as in \eqref{eqn:RRP} and associated rational inner functions $\phi$ by starting with an $n$-vertex, undirected, simple graph $G$ with vertices $v_1,\dots, v_n$, letting $A$ denote its adjacency matrix, and letting $Y$ be a given coloring matrix, defined by 
\[ A_{ij} = \left\{ \begin{array}{ll} 1 & \text{if there is an edge between $v_i$ and $v_j$} \\
 0 & \text{if there is no edge between $v_i$ and $v_j$} \end{array} \right. \quad \text{ and } \quad Y =   \begin{bmatrix} 
1 &&& \\ & \ddots && \\
&& 1 & \\
&&&t \end{bmatrix}, \]
where $t \in [0,1)$. Here, the matrix $Y$ should be interpreted as coloring vertices $v_1$ through $v_{n-1}$ with blue and vertex $n$ with a color ranging from red ($t=0$) to shades of purple ($0<t<1$). Other ``colorings'' could be studied using different choices of $Y$. We let $\alpha = e_1$ as before, which somewhat focuses our study on the first vertex of $G.$

For these choices of $A, Y, \alpha,$ we denote $f$ in \eqref{eqn:RRP} by $f_A^t$ and the associated function $\beta \circ f_A^t \circ \beta^{-1}$ by $\phi_A^t = \frac{q_A^t}{p_A^t}$ (with no common terms between numerator and denominator polynomials $q_A^t$ and $p_A^t$). With the choice of $\alpha=e_1$, observe that 
\begin{equation} \label{eqnft} f_A^t(z_1, z_2) = \alpha^\star (A-z_1 Y - z_2(I-Y))^{-1} \alpha =  (A - z_{Y})^{-1}_{11},\end{equation}
the $(1,1)$-entry of the matrix inverse of $A- z_{Y}$. If $t=0$, we generally drop the $t$-dependence and write $f_A, \phi_A, q_A,$ and $p_A$. We often denote the graph $G$ with adjacency matrix $A$ by $G_A.$

 This construction was used to good effect in \cite{bps18, bps19a, bps19b,Pascoe18}, where it yielded key examples that illustrated various stable polynomial and rational inner function behavior. In this paper, we begin a study of the properties of the initial  graph $G_A$ and the resulting rational inner functions $\phi_A^t$ and their stable polynomial denominators $p_A^t$, which explains many of the behaviors observed in \cite{bps18, bps19a, bps19b}. 
The interested reader should also consult the new preprint \cite{ATTT23} by Adlin, Thai, Tiscarenoby, and Tully-Doyle, which explores the structure of the rational functions in \eqref{eqnft} using a number of creative graph-theoretic and linear algebra techniques. 

Other stable polynomial constructions appear in the literature. For example in \cite{GW06}, Geronimo and Woerdeman give matrix positivity conditions that allow one to check whether a given polynomial is stable and also provide an algorithm for producing stable $p$ with some specified zeros. Meanwhile, in \cite[Theorem 2.19]{bpsk22}, the authors construct examples of stable polynomials with particular Puiseux expansions at a specified boundary zero. Recently, in \cite{Sola23}, Sola constructs stable polynomials on the polydisk with certain behaviors at boundary zeros using iterates of rational inner functions. In all constructions,  the key point of concern is whether the polynomial has zeros on $\mathbb{T}^2$ and how the polynomial behaves near such boundary zeros.

\subsection{Main Results \& Paper Outline} In this paper, we assume that we are starting with an undirected, simple graph $G_A$ with vertices $v_1, \dots, v_n$ and then investigate the connections between the graph $G_A$ and the associated stable polynomials $p_A$ and $p_A^t$ for $t\in(0,1).$ Specifically, in Section \ref{sec:bdyzero}, we address the first natural question one should ask when given a class of stable polynomials on $\mathbb{D}^2$: 
\begin{center} \emph{Question 1: When does a stable polynomial in the class have a boundary zero?}\end{center} 
As encoded in the following result, the basic structure of the graph $G_A$  immediately  tells us whether $p_A$ depends on  both variables and whether it has a zero on $\mathbb{T}^2$:

\begin{theorem} \label{thm:1} Let $G_A$ be an $n$-vertex graph and let $p_A$ be the associated stable polynomial. Then the following hold:
\begin{itemize}
    \item[i.] If $v_1$ is isolated, then $p_A$ is constant. 
    \item[ii.] If $v_1$ is neither isolated nor path connected to $v_n$, then $p_A \in \mathbb{C}[z_1]$ is not constant and has no zeros on $\mathbb{T}^2$.
    \item[iii.] If $v_1$ and $v_n$ are path connected, then $p_A \in \mathbb{C}[z_1, z_2]$ depends on $z_2$ and $p_A(-1,1)=0$.
\end{itemize}
\end{theorem}

Subsection \ref{sec:lemmas} includes some straightforward but useful lemmas. Then Theorem \ref{thm:1} is proved in Subsection \ref{subsec:thm1} as a sequence of shorter results.

Graphs illustrating the three cases in Theorem \ref{thm:1} are given in Figure \ref{fig:threegraph}. As an example of Theorem \ref{thm:1}, one can check that the polynomial $p$ from \eqref{eqn:firstp} satisfies $p(-1,1)=0$.  Additionally, note that any $p_A$ must satisfy exactly one of these conditions: $p_A$ is constant, $p_A \in \mathbb{C}[z_1]$ is not constant, and $p_A \in \mathbb{C}[z_1,z_2]$ has some $z_2$ dependence. Thus, Theorem \ref{thm:1} gives these double implications: $p_A$ is  constant if and only if $v_1$ is isolated, $p_A$ has only $z_1$ dependence if and only if $v_1$ is not isolated and not path connected to $v_n$, and $p_A$ has $z_2$ dependence if and only if $v_1$ and $v_n$ are path connected. 

\begin{figure}[ht!] 
\begin{minipage}{0.30\linewidth}
\begin{center}
\begin{tikzpicture}[
roundnode/.style={circle, draw=blue!60, fill=blue!5, very thick, minimum size=7mm},
squarednode/.style={rectangle, draw=red!60, fill=red!5, very thick, minimum size=7mm}, 
]
\node[roundnode]      (node1)                              {1};
\node[roundnode]        (node2)       [below=of node1] {2};
  \node[roundnode]    (node3)       [right=of node2] {3};
 \node[squarednode]       (node4)       [right=of node1] {4};
\draw[-] (node2.east) -- (node3.west);
\draw[-]  (node4.south) -- (node3.north);
\end{tikzpicture}
\end{center}
\end{minipage}%
\begin{minipage}{0.30\linewidth}
\begin{center}
\begin{tikzpicture}[
roundnode/.style={circle, draw=blue!60, fill=blue!5, very thick, minimum size=7mm},
squarednode/.style={rectangle, draw=red!60, fill=red!5, very thick, minimum size=7mm}, 
]
\node[roundnode]      (node1)                              {1};
\node[roundnode]        (node2)       [below=of node1] {2};
  \node[roundnode]    (node3)       [right=of node2] {3};
 \node[squarednode]       (node4)       [right=of node1] {4};
\draw[-] (node1.south) -- (node2.north);
\draw[-]  (node4.south) -- (node3.north);
\end{tikzpicture}
\end{center}
\end{minipage}%
\begin{minipage}{0.30\linewidth}
\begin{center}
\begin{tikzpicture}[
roundnode/.style={circle, draw=blue!60, fill=blue!5, very thick, minimum size=7mm},
squarednode/.style={rectangle, draw=red!60, fill=red!5, very thick, minimum size=7mm}, 
]
\node[roundnode]      (node1)                              {1};
\node[roundnode]        (node2)       [below=of node1] {2};
  \node[roundnode]    (node3)       [right=of node2] {3};
 \node[squarednode]       (node4)       [right=of node1] {4};
\draw[-] (node1.south) -- (node2.north);
\draw[-] (node2.east) -- (node3.west);
\draw[-]  (node4.west) -- (node1.east);
\draw[-]  (node4.south) -- (node3.north);
\end{tikzpicture}
\end{center}
\end{minipage}%

\caption{Examples of the three situations in Theorem \ref{thm:1} with $n=4$} 
\label{fig:threegraph}
\end{figure}
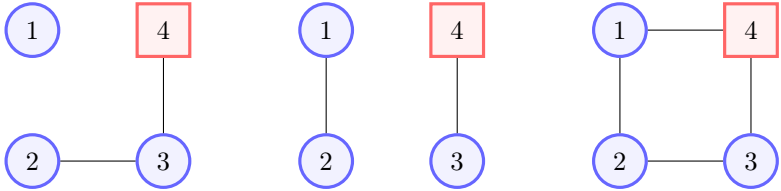

Because $p_A$ is the denominator of a rational inner function, if it has a zero on $\mathbb{T}^2$, then it must have both $z_1$ and $z_2$ dependence. Thus, Theorem \ref{thm:1} implies that if $p_A$ has any zero on $\mathbb{T}^2$, then it must possess some $z_2$ dependence and so, be zero at $(-1,1).$ Many polynomials in the class under consideration can have additional boundary zeros; for example, the polynomial in \eqref{eqn:firstp} also satisfies $p(1,1)=0.$ Classes of special special graphs and their polynomials were studied by the second author in \cite{H23}, where such additional boundary zeros were identified.

Although Theorem \ref{thm:1} is stated when $t=0$ (when the last vertex is colored red), change-of-variables arguments give the following:

\begin{corollary} \label{cor:pt} Let $G_A$ be an $n$-vertex graph and let $p^t_A$ be the associated stable polynomial with $t\in(0,1)$. If  $v_1$ and $v_n$ are path connected, then $p^t_A(-1,-1)=0$. If $v_1$ and $v_n$ are not path connected, then $p_A^t$ has no boundary zeros.
\end{corollary}

For example, using the graph in Figure \ref{fig:4vertexgraph} with $t \in (0,1)$ gives the stable  polynomial
\[ \begin{aligned} p_A^t(z_1, z_2) &= 4 +4z_1 -5tz_1-tz_1^2 -3tz_1^3+tz_1^4 -z_2 +5tz_2 -2 z_1 z_2 +t z_1 z_2 
\\ &\hspace{.2in} - 4 z_1^2 z_2 + 3 t z_1^2 z_2 -2 z_1^3 z_2 -t z_1^3 z_2 + z_1^4 z_2, \end{aligned}\]
which satisfies $p^t_A(-1,-1) =0.$ Corollary \ref{cor:pt} is proved in Subsection \ref{transition}. In that subsection, we also track how general boundary zeros of $p_A$ and $p_A^t$ are related, see Proposition \ref{thm:8}. 

The second natural question one should ask when given a class of stable polynomials with boundary zeros is: 
\begin{center} \emph{Question 2: How do the polynomials' zero sets behave near the guaranteed boundary zero(s)?}\end{center} 
For stable polynomials on $\mathbb{D}^2$, their zero set behavior is often described using a positive, even integer called \emph{contact order}, which measures how the function's zero set approaches $\mathbb{T}^2$. 
As proved in \cite{bps18, bps19a, bpsk22},
 the contact order $K$ governs the behavior of the associated rational inner function $\phi$'s unimodular level sets,  derivative integrability, and local non-tangential polynomial approximations. While most stable polynomials (for example the one in \eqref{eqn:firstp}) have contact order $K=2$ at their boundary zeros, this is not always true. 

In Section \ref{sec:CO}, we study the contact order of stable polynomials $p_A^t$ constructed from graphs. Subsection \ref{subsec:COdef} gives important information about defining, studying, and computing contact order. Meanwhile, Subsection \ref{subsec:COtvalues} notes that for \emph{most} boundary zeros of $p_A$, the associated zero of $p_A^t$ has the same contact order; this appears as Proposition \ref{prop:COt}.

In Subsection \ref{subsec:COpaths},  we consider polynomials constructed from \emph{path graphs}, which are graphs where each vertex is only connected to the following vertex. For example, the graph in Figure \ref{fig:4vertexgraph} would be a path graph if we removed the edge between $v_1$ and $v_4.$ We prove the following:

\begin{theorem} \label{thm:COpaths} Let $G_A$ be an $n$-vertex path graph. Then $p_A$ has contact order $K=2(n-1)$ at its boundary zero $(-1,1)$.
\end{theorem}

Additionally, as part of the proof, we obtain a concrete formula for $f_A$. This result also aligns with the work of Pascoe in \cite{Pascoe18}. Specifically, the final example in that paper shows that rational functions \eqref{eqnft} constructed with path graphs and  $t\in (0, 1)$ have particular non-tangential regularity at infinity. Later work in \cite{bps18, bpsk22} showed that this non-tangential regularity measures the contact order (in this case, $K=2n$)  of the associated stable polynomial at $(-1,-1)$. So, a $t\ne 0$ version of Theorem \ref{thm:COpaths} follows from earlier work of Pascoe and collaborators. Our Theorem \ref{thm:COpaths} shows that a similar result (with a drop in contact order) holds when $t=0$ and aligns with what appeared in examples in \cite{bps18, bpsk22}.

For path graphs, $n-1$ is also the length of the shortest path connecting $v_1$ and $v_n$.  We conjecture that this quantity more generally controls the contact order of $p^t_A$ at the guaranteed boundary zero $(-1,1)$ for $t=0$ and $(-1,-1)$ for $t \in (0,1).$

\begin{conjecture} \label{con:CO} Let $G_A$ be an $n$-vertex graph and let $\ell$ denote the length of the shortest path connecting vertices $v_1$ and $v_n$ in $G_A$. If $t=0$, then $p_A$ has contact order $2\ell$ at $(-1,1)$ and if $t \in (0,1)$, then $p_A^t$ has contact order $2\ell+2$ at $(-1,-1).$
\end{conjecture}

To investigate this conjecture, in Subsection \ref{subsec:GM}, we study how the contact order of $p_A$ at $(-1,1)$ changes when we append additional vertices and edges to the graph $G_A$. The results are summarized in the following theorem:

\begin{theorem} \label{thm:CO2} Let $G_A$ be an $n$-vertex graph such that $p_A$ has contact order $K$ at $(-1,1).$ 
\begin{itemize}
    \item[i.] Create $G_{\widehat{A}}$ from $G_A$ by adding a vertex $v_{n+1}$ that is only connected directly to $v_n.$ Then $p_{\widehat{A}}$ has contact order $K+2$ at $(-1,1).$
    \item[ii.] Create $G_{\widetilde{A}}$ from $G_A$ by adding a vertex $v_{0}$ that is only connected directly to $v_1$ (and then renumbering to $v_1 ,\dots, v_{n+1}$). Then $p_{\widetilde{A}}$ has contact order $K+2$ at $(-1,1).$
      \item[iii.] Assume that $v_n$ is only connected to the rest of $G_A$ by an edge to $v_{n-1}.$  For $m \ge n+1$, create an $m$-vertex graph $G_{\breve{A}}$ from $G_A$ by  attaching an $(m-n)$-vertex graph $G'$ to $G_A$ via an edge to $v_{n-1}$  and then relabeling $v_n$ to $v_m$. Then $p_{\breve{A}}$ has contact order $K$ at $(-1,1).$   
\end{itemize}
\end{theorem}

Note that all three situations in Theorem \ref{thm:CO2} support Conjecture \ref{con:CO}. In the first two situations, we extend the length of the shortest path connecting the first and last vertices by $1$; thus, it makes sense that the contact order increases by $2$. In the last situation, the length of the shortest path connecting the first and last vertices does not change, so we would expect the contact order to remain the same. Theorem \ref{thm:CO2} is proved as three separate results (Theorem \ref{thmCO:i}, Theorem \ref{thmCO:ii}, Theorem \ref{thmCO:iii}), where Theorem \ref{thmCO:iii} also allows the option of constructing $G_{\breve{A}}$ by attaching $G'$ to both $v_{n-1}$ and $v_n.$
%The preprint \ref{con:CO} includes other investigations of Conjecture \ref{con:CO} and a variety of interesting related results. 

It is worth noting that since the completion of this work, the $t=0$ part of Conjecture \ref{con:CO} has been resolved by  Adlin, Thai, Tiscarenoby, and Tully-Doyle using very interesting and efficient methods in the new preprint \cite{ATTT23}. The $t\ne 0$ case appears to still be open.

Finally, Section \ref{sec:open} contains a short list of open questions related to stable polynomials and graphs. Of particular interest to the authors are investigations that go in the other direction: given a stable polynomial or Cauchy transform function  associated to a colored graph, what can one say about the original graph? To the best of our knowledge, little in that direction is currently known.

\section*{Acknowledgments}
Special thanks to Ryan Tully-Doyle and J.E. Pascoe for very useful insights and discussions during the writing of this paper. The authors Kelly Bickel and Yang Hong were partially supported by the National Science Foundation DMS grant \#2000088 during these research activities. Additional support was provided to Hong by Bucknell University. Many results in Section $2$ originally appeared in Hong's Bucknell honors thesis \cite{H23}.

\section{Existence of Boundary Zeros} \label{sec:bdyzero}
\subsection{Key Lemmas}\label{sec:lemmas}
In this section, we record some facts about inverses, determinants, and the functions $f^t_A$ defined in \eqref{eqnft} that will be used later.

\begin{lemma}\label{lem:one-one elem of inverse matrix}
If $A$ is a square matrix and $\overline{A}$ is the matrix obtained from $A$ by removing its first row and column, then the $(1,1)$ entry of $A^{-1}$ is given by \[(A^{-1})_{11} = \frac{\det{\overline{A}}}{\det{A}}.\]
\end{lemma}

This is an immediate consequence of the adjugate formula for $A^{-1}.$ We will also require the following. 

\begin{lemma} \label{lem:Detmatrix} Let $M$ be a square matrix, partitioned as 
\[
M=
\left[
\begin{array}{c|c}
A  & B \\
\hline
C & D
\end{array}
\right],\]
where $A$ and $D$ are square matrices and $A^{-1}$ and $D^{-1}$ exist. Then 
\begin{equation}\label{det of partitioned matrix}
    \det(M) = \det(A)\det(D - CA^{-1}B) = \det(D)\det(A - BD^{-1}C).
\end{equation}
\end{lemma}

This is well known and can be found for example in (50) in Example 47 of \cite{Garcia2020}. We lastly record the following fact about the functions $f^t_A$ from \eqref{eqnft}:

\begin{lemma} \label{lem:1} Let $G_A$ be an $n$-vertex graph with $n \ge 3$. Fix $1 < i,j <n$ and let $G_{A'}$ be the graph obtained from $G_A$ by switching $v_i$ and $v_j$. Then for $t \in [0,1),$ $f^t_A = f^t_{A'}$.
\end{lemma}

\begin{proof} Let $U$ be the unitary matrix obtained from the $n \times n$ identity matrix $I$ by switching its $i^{th}$ and $j^{th}$ columns and note that $U^{-1} =U$. Then the adjacency matrices satisfy $A' = U A U$ and the formula for $z_Y$ gives $z_Y = U z_YU$. Thus,
 \[ (A'-z_Y)^{-1} =(UA U-z_Y)^{-1}= U (A - U z_Y U)^{-1} U = U (A - z_Y)^{-1} U.\]
As the application of $U$ does not alter the $(1,1)$ entry of $(A - z_Y)^{-1}$, \eqref{eqnft} implies that $f^t_A =f^t_{A'}.$
\end{proof}

\subsection{Proof of Theorem \ref{thm:1}} \label{subsec:thm1}
This section provides the proof of Theorem \ref{thm:1}. We first establish parts (i) and (ii) as the following proposition. The functions $\beta, \beta^{-1}, f_A, \phi_A,$ and  $p_A$ are defined as in the introduction with $t=0$.

\begin{proposition} \label{prop:Paz1} If $v_1$ and $v_n$ are not path connected, then $p_A \in \mathbb{C}[z_1]$. Furthermore, if $v_1$ is isolated, then $p_A$ is constant and if $v_1$ is not isolated, then $p_A$ is not constant.
\end{proposition} 

\begin{proof} By Lemma \ref{lem:1}, we can reorder the vertices of $G_A$ so that $v_1, \dots, v_m$ are exactly the vertices in the path-connected component of $G_A$ containing $v_1.$ Then $A$ can be partitioned as 
\[A = 
\left[
\begin{array}{c|c}
A_1 & 0 \\
\hline
0 & \star
\end{array}
\right], \ \text{ so } \ (A - z_Y)^{-1} = \left[
\begin{array}{c|c}
(A_1-z_1I)^{-1} & 0 \\
\hline
0 & \star
\end{array}
\right] ,
\]
where $A_1$ is an $m\times m$ matrix, $0$ and $I$ denote the zero and identity matrices of the appropriate size, and $\star$ denotes matrices immaterial to our calculations. From this, it is clear that the $(1,1)$ entry of $(A - z_Y)^{-1}$, and hence $f_A$, depends only on $z_1,$ so $p_A \in \mathbb{C}[z_1]$. 
Furthermore if $m=1$, then $f_A = -\frac{1}{z_1}$, which implies $\phi_A = -z_1$ and $p_A$ is constant. Moreover, the rational inner function form from \eqref{eqn:RIFform} implies that if $f_A$ only depends on $z_1$ and $p_A$ is constant, then $\phi_A = \alpha z_1^k$, for some $\alpha \in \mathbb{T}$, $k \in \mathbb{N}.$ That implies
\[ f_A(z_1) = \beta^{-1} \circ \phi_A \circ \beta   = -i + \frac{2i}{1 + \alpha (-1 + \frac{2i}{i+z_1})^k },\]
and if we expand this around $\infty$ for $z_1$ sufficiently large, we get
\begin{equation} \label{eqn:fAz1} f_A(z_1)= \left(-i + \frac{2 i}{\alpha (-1)^k+1}\right) - \frac{4(-1)^k \alpha k}{(1 + (-1)^k \alpha)^2}\frac{1}{z_1} + \text{higher order terms.}\end{equation}
Meanwhile, we can write $f_A$ using a Neumann series (for sufficiently large $z_1$) as
\[
\begin{aligned}
f_A(z_1) &=(A_1-z_1I)^{-1}_{11} = -\frac{1}{z_1}(I - \tfrac{1}{z_1}A_1 )^{-1}_{11}  \\ 
&= \sum_{n=0}^{\infty} -\frac{1}{z_1^{n+1}} (A_1^n)_{11} = 0 -\frac{1}{z_1} + \text{higher order terms.}\end{aligned}\]
Setting the first coefficient in \eqref{eqn:fAz1} to $0$ and the second to $-1$ implies that $\alpha =-1$ and $k=1$. Thus $f_A = -\frac{1}{z_1}$ is the only function in our paradigm yielding a constant $p_A$. 

Now, assume $v_1$ is not isolated. Thus, $v_1$ is connected to some other vertex $v_j$ and so, there are paths of length $2$ that begin and end at $v_1$. Because $(A^2)_{11}$ is the number of paths of length $2$ that begin and end at $v_1$, we know $(A^2)_{11} \ne 0$. Thus by the above expansion of $f_A$, if $v_1$ is not isolated, then $f_A \ne - \frac{1}{z_1}$, which means $p_A$ is not constant.
\end{proof}

To prove part (iii) of Theorem \ref{thm:1}, we require the following lemma, which gives us an easy way to use $f_A$ to identify when $p_A$ has a boundary zero at $(-1,1)$.

\begin{lemma} \label{lem:2} Define $h_A$ by $h_A(z_1,z_2) = f_A(-\tfrac{1}{z_1}, z_2)$ and write $h_A = \frac{q_1}{q_2},$ where $q_1, q_2$ are polynomials with no common factors. If $q_1(0,0)=0=q_2(0,0)$, then $p_A(-1,1)=0$.
\end{lemma}

\begin{proof} Note that for almost every $(z_1, z_2) \in \mathbb{C}^2$, $\phi_A$ satisfies
\[ 
\begin{aligned}
\phi_A(z_1, z_2) &= \beta \circ h_A\left( \frac{-1}{\beta^{-1}(z_1)}, \beta^{-1}(z_2) \right)\\
&= \frac{ q_2\left( -\tfrac{1}{\beta^{-1}(z_1)}, \beta^{-1}(z_2) \right) +iq_1\left( -\tfrac{1}{\beta^{-1}(z_1)}, \beta^{-1}(z_2) \right)}{q_2\left( -\tfrac{1}{\beta^{-1}(z_1)}, \beta^{-1}(z_2) \right)-iq_1\left( -\tfrac{1}{\beta^{-1}(z_1)}, \beta^{-1}(z_2) \right)}.\end{aligned}\]
Multiply through by the lowest powers of $(z_1-1)$ and $(z_2+1)$ to turn both the numerator and denominator into polynomials and denote those polynomials by $r_A, s_A$. Then
\[ \phi_A(z_1, z_2) = \frac{q_A(z_1, z_2)}{p_A(z_1, z_2)} = \frac{r_A(z_1, z_2)}{s_A(z_1, z_2)},\]
for almost every $(z_1, z_2) \in \mathbb{C}^2.$ By way of contradiction, assume $p_A(-1,1) \ne 0$. By our assumptions, $r_A(-1,1) = 0 = s_A(-1,1),$ which implies that $r_A, s_A$ must share a common, non-constant factor. This implies that $q_1, q_2$ share a common curve of zeros as well and hence, a common factor, which gives the contradiction. \end{proof}

Now we can prove part (iii) of Theorem \ref{thm:1}. 

\begin{proposition}\label{prop:pConn-genZero} If $v_1$ and $v_n$ are path connected, then $p_A \in \mathbb{C}[z_1, z_2]$ depends on $z_2$  and $p_A(-1,1)=0$.
\end{proposition}

\begin{proof} Assume that $v_1$ and $v_n$ are path connected. Define the submatrices $A_1$, $A_2$, $A_3$ and vectors $\alpha, \gamma, \zeta, \eta, b$ of the appropriate size, so that  
\[
A=
\left[
\begin{array}{c|c}
0 & \alpha^T\\
\hline
\alpha & A_1
\end{array}
\right] =
\left[
\begin{array}{c|c}
A_2  & \gamma \\
\hline
\gamma^T & 0
\end{array}
\right]
=
\left[
\begin{array}{c|c|c}
0 & \zeta^T & b \\
\hline
\zeta & A_3 & \eta \\
\hline
b & \eta^T & 0
\end{array}
\right].
\]
Then using Lemma \ref{lem:one-one elem of inverse matrix} and Lemma \ref{lem:Detmatrix}, we obtain
\[ f_A(z_1, z_2) = \frac{\det(A_1 - z_Y)}{\det(A - z_Y)} = \frac{\det(A_3 - z_1  I)(-z_2 - \eta^T(A_3 - z_1I)^{-1}\eta)}{\det(A_2 - z_1  I)(-z_2 - \gamma^T(A_2 - z_1I)^{-1}\gamma)}.\]
With an eye towards using Lemma \ref{lem:2}, replace $z_1$ with $-\tfrac{1}{z_1}$ and multiply by $\frac{z_1^{n-1}}{z_1^{n-1}}$ to obtain
\begin{equation} \label{eqn:hg}
h_A(z_1,z_2):=f_A(-\tfrac{1}{z_1}, z_2) 
=  \frac{z_1 \det(z_1A_3 + I)(-z_2 - z_1\eta^T(z_1A_3 +I)^{-1}\eta)}{\det(z_1A_2 +  I)(-z_2 - z_1\gamma^T(z_1A_2 + I)^{-1}\gamma)}.
\end{equation}
We claim that the factors containing $z_2$ in the numerator and denominator in \eqref{eqn:hg} cannot cancel out. By way of contradiction, assume that they do. Then for all $z_1$,
\[ \gamma^T \left(z_1A_2 +  I\right)^{-1} \gamma  = \eta^T \left(z_1A_3 + I\right)^{-1}\eta. \]
For small $z_1$, the Neumann series for $-z_1 A_1$ and $-z_1A_2$ converge and so, we can conclude
\[ \sum_{m=0}^{\infty} (-z_1)^{m}  \gamma^T A_2^m \gamma   =  \gamma^T \left(z_1A_2 + I\right)^{-1} \gamma  = \eta^T \left(z_1A_3 + I\right)^{-1}\eta = \sum_{m=0}^{\infty} (-z_1)^m \eta^T A_3^m \eta,\]
or more specifically, for all $m \ge 0,$
\begin{equation} \label{eqn:mequal}  \gamma^T A_2^m \gamma  = \eta^T A_3^m \eta.\end{equation}
Setting $m = 0$, this immediately gives $\gamma^T \gamma  = \eta^T \eta$, so $b=0.$ This implies that $v_1$ and $v_n$ are not connected by a single edge. But since they are path connected, $v_1$ must be path connected to some other vertex through a path that does not include $v_n$, such that this intermediate vertex is directly connected to $v_n$. By Lemma \ref{lem:1}, we can assume that this intermediate vertex is $v_{n-1}$.

Let $M$ denote the length of a path between $v_1$ and $v_{n-1}$ in $G_A$ that does not go through $v_n$. By well-known properties of adjacency matrices, $k:=(A_2)^M_{1,n-1}$ gives the number of paths between $v_1$ and $v_{n-1}$ of length $M$ in the graph obtained by removing $v_n$ from $G_A$. Then we can write
\[
A_2^M=
\left[
\begin{array}{c|c}
c & \begin{matrix}\epsilon^T & k \end{matrix} \\
\hline
\begin{matrix} \epsilon \\ k \end{matrix} & B + A_3^M \\
\end{array}
\right],
\]
where $k>0$, $\epsilon$ is a vector, and $B$ is a matrix, both with real, nonnegative entries. Then since $k>0$, $b=0$, and $\eta_{n-2} =1$ because $v_{n-1}$ and $v_n$ are path connected, we have
\[\gamma^T A_2^{2M} \gamma = \eta^T \left(  \begin{bmatrix} \epsilon \\ k \end{bmatrix} \begin{bmatrix}\epsilon^T & k \end{bmatrix} + (B + A_3^M)^2 \right ) \eta >\eta^T A_3^{2M} \eta,\]
which contradicts \eqref{eqn:mequal}.  Thus, the factors in the numerator and denominator that involve $z_2$ in \eqref{eqn:hg} do not cancel. 

Since those factors vanish at $(0,0),$ we can write $h_A = \frac{q_1}{q_2}$ where $q_1$ and $q_2$ are polynomials with no common factors and $q_1(0,0) =0=q_2(0,0).$ Thus, by Lemma \ref{lem:2}, $p_A(-1,1)=0$. Since $p_A(-1,1) =0$, it follows immediately that $p_A$ must have both $z_1$ and $z_2$-dependence, since if $\phi_A$  were one variable, it would be a finite Blaschke product, and those do not have singularities on $\mathbb{T}.$
\end{proof}

\subsection{Boundary zeros for nonzero t}\label{transition}

In this subsection, we track how the location of  boundary zeros of $p_A^t$ depends on the coloring parameter $t$. This allows us to use the $t=0$ case to derive results when $t \in (0,1)$ and in particular, prove Corollary \ref{cor:pt}.

We require this initial lemma. (Note the superscript $n$ in $z_j^n$  and $w_j^n$ below represents an index and not a power.)

\begin{lemma}\label{lem:phi_seq} Let $\phi$ be a rational inner function on $\mathbb{D}^2$ with stable polynomial denominator $p$ (after any common factors are canceled). Then

\begin{enumerate}
    \item[i.] If $p(\tau_1,\tau_2) = 0$ for $(\tau_1,\tau_2) \in \mathbb{T}^2$, then for almost every $\eta \in \mathbb{T}$, there is a sequence $z_n = (z_1^n,z_2^n) \subset \mathbb{T}^2$, with each $z_1^n \ne \tau_1$, each $z_2^n \ne \tau_2$, each $\phi(z_1^n,z_2^n) = \eta$, and $(z_1^n,z_2^n) \rightarrow (\tau_1, \tau_2)$.
    \item[ii.] Assume there are distinct $\eta_1, \eta_2 \in \mathbb{T}$ with corresponding sequences $z_n = (z_1^n,z_2^n) \subset \mathbb{T}^2$ and $w_n = (w_1^n,w_2^n) \subset \mathbb{T}^2$ such that $(z_1^n,z_2^n)$ and $(w_1^n,w_2^n)$ both converge to $(\tau_1,\tau_2) \in \mathbb{T}^2$ and $\phi(z_1^n,z_2^n) = \eta_1$ and $\phi(w_1^n,w_2^n) = \eta_2$ for all $n$. Then $p(\tau_1,\tau_2) = 0$.
\end{enumerate}
\end{lemma}

\begin{proof} For (i), note that since $\phi$ is a rational inner function, $p(\tau_1,\tau_2) = 0$ implies that $\phi$ does not have a continuous unimodular extension to a neighborhood of $(\tau_1, \tau_2)$ in $\mathbb{T}^2$; this would imply that $\phi$ actually extended analytically to a neighborhood of $(\tau_1,\tau_2)$ in $\mathbb{C}^2$, see Theorem 1.5 in \cite{kb13}, which in turn would contradict the fact that the numerator and denominator of $\phi$ have no common factors. Then the statement about level set sequences follows from an application of Corollary 1.7 in \cite{Pascoe17} and the description of these level sets (in terms of non-constant analytic functions for almost every $\eta$) in  Theorem 2.8  in \cite{bps19a}. For (ii), note that the assumptions imply that $\phi$ is discontinuous at $(\tau_1,\tau_2)$ and since $\phi$ is rational, its denominator must vanish at $(\tau_1, \tau_2).$
\end{proof}

Now we study the correspondence between boundary zeros in the $t = 0$ case and in the $t \neq 0$ case. We first study the boundary zero whose existence is often guaranteed by Theorem \ref{thm:1}.
\begin{proposition} \label{thm:7} Using the notation defined earlier, if $p_A(-1,1) = 0$, then $p_A^t(-1,-1) = 0$.
\end{proposition}

\begin{proof} Assume $p_A(-1,1) = 0$. By Lemma \ref{lem:phi_seq}, there exist  distinct $\eta_1, \eta_2 \in \mathbb{T}$ with $\eta_1, \eta_2 \neq -1$ that possess corresponding sequences $z_n = (z_1^n,z_2^n)$ and $ w_n= (w_1^n,w_2^n)$ in  $\mathbb{T}^2$  such that $z_n, w_n \rightarrow (-1,1)$ with $\phi_A(z_1^n,z_2^n) = \eta_1$, $\phi_A(w_1^n,w_2^n) = \eta_2$, and $z_1^n, w_1^n \ne -1$  for each $n$. Recall the definitions of $\beta, \beta^{-1}$ from \eqref{eqn:beta} and
 define sequences $r_n, s_n \subseteq \mathbb{R}^2$ by  $$r_n = (r_1^n,r_2^n) = (\beta^{-1}(z_1^n),\beta^{-1}(z_2^n)) \ \ \text{ and } \ \ s_n = (s_1^n,s_2^n) = (\beta^{-1}(w_1^n),\beta^{-1}(w_2^n))$$
 and note that $r_n, s_n \rightarrow (\beta^{-1}(-1),\beta^{-1}(1)) = (\infty, 0)$, in the sense that the sequences with elements $|\beta^{-1}(r^1_n)|$ and $ |\beta^{-1}(s^1_n)|$ diverge to infinity. Set $t_1 = \beta^{-1}(\eta_1) \in \mathbb{R}$ and $t_2 = \beta^{-1}(\eta_2) \in \mathbb{R}$. Since $f_A = \beta^{-1} \circ \phi_A \circ \beta$, we have $f_A (r_n) = t_1$, $f_A (s_n) =t_2$ for all $n.$

Now we shift these facts to the $t \in (0,1)$ setting.
By the definition from \eqref{eqnft}, we have $$f_A^t(z_1,z_2) = f_A\big(z_1,tz_1 + (1-t)z_2\big).$$ Construct a sequence $\hat{r}_n \subseteq \mathbb{R}^2$ from $r_n$ by 
$$\hat{r}_1^n = r_1^n, \;\;\; \hat{r}_2^n = \frac{1}{1-t}(r_2^n - tr_1^n)$$ 
and  construct $\hat{s}_n$ from $s_n$ in an identical way.
Then  $f_A^t(\hat{r}_n) =  f_A(r_n)=t_1$, $f_A^t(\hat{s}_n) = f_A(s_n)=t_2$ for all $n$, and 
 $\hat{r}_n,\hat{s}_n \rightarrow (\infty,\infty)$.

To convert this to the $\phi_A^t$ setting,  consider the sequences $\hat{z}_n = (\beta(\hat{r}^1_n), \beta(\hat{r}^2_n))$ 
and $\hat{w}_n = (\beta(\hat{s}^1_n), \beta(\hat{s}^2_n))$ in $\mathbb{T}^2.$ Then $\hat{z}_n,\hat{w}_n \rightarrow (-1,-1)$ and $\phi_A^t(\hat{z}_1^n,\hat{z}_2^n) = \eta_1$, $\phi_A^t(\hat{w}_1^n,\hat{w}_2^n) = \eta_2$ for all $n$. An application of Lemma \ref{lem:phi_seq} gives $p_A^t(-1,-1) = 0$, as needed.
\end{proof}

We can now prove Corollary \ref{cor:pt}.

%\begin{theorem} \label{col:1} If in an n-vertex graph $G_A$, vertex $1$ and vertex $n$ are path connected, then $p_A(-1,-1) = 0$ when $t \ne 0$.
%\end{theorem}

\begin{proof} Let vertices $v_1$ and $v_n$ be path connected. Then Theorem \ref{thm:1} immediately implies that $p_A(-1,1)=0$ and so, Proposition \ref{thm:7} tells us that $p_A^t(-1,-1)=0.$

Conversely, if $v_1$ and $v_n$ are not path connected, then arguments identical to those at the beginning of Proposition \ref{prop:Paz1} imply that $f_A^t$ (and hence $\phi_A^t$) only depends on $z_1$. Then $\phi_A^t$ is a finite Blaschke product, which cannot have any boundary singularities, and so its denominator $p_A^t$ cannot have any boundary zeros. 
\end{proof}

It is worth noting that one can also use the previous change-of-variables arguments to move between general boundary zeros of $p_A$ and general boundary zeros of $p_A^t$.

\begin{proposition} \label{thm:8} Using the notation defined earlier, for any $\tau_1,\tau_2 \in \mathbb{T}$ with $\tau_1,\tau_2 \neq -1$, $p_A(\tau_1,\tau_2) = 0$ if and only if $p_A^t(\lambda_1,\lambda_2) = 0$ where $\lambda_1 = \tau_1$ and $\lambda_2 = \beta\left(\dfrac{\beta^{-1}(\tau_2) - t\beta^{-1}(\tau_1)}{1 - t}\right)$.
\end{proposition}

\begin{proof}

The forward direction is proved in a way analogous to that of Proposition \ref{thm:7}.

For the backwards direction, the restriction that $\tau_1, \tau_2 \ne -1$ means that the limits of the sequences $\hat{r}_n, \hat{s}_n$ also determine the limits of the $r_n, s_n$ sequences (since none of the sequences diverge to $\infty$). Then we can simply reverse the steps in the proof Proposition \ref{thm:7}.
\end{proof}

\section{Contact Order} \label{sec:CO}
\subsection{Computing Contact Order} \label{subsec:COdef}
Let $\phi$ be a rational inner function with stable polynomial denominator $p$ and a boundary singularity at $(1,1)$. Recall the definition of $\tilde{p}$ from \eqref{eqn:tildep}.
%Let the bidegree of $p$ be $(m,n)$ and let $\tilde{p}$ denote the reflection of $p$, given by
%\[ \tilde{p}(z_1, z_2) = z_1^mz_2^n \overline{ p\left(\tfrac{1}{\bar{z}_1}, \tfrac{1} {\bar{z}_2} \right)}.\]
%Up to multiplication by a monomial $ \eta z_1^\ell z_2^k$ for $\eta \in \mathbb{T}$ and $\ell, k \in \mathbb{N}$, $\tilde{p}$ is exactly the numerator of $\phi$. 
Then
the contact order of $\phi$ at $(1,1)$ (or equivalently, the contact order of $p$ at $(1,1)$) is a positive even integer $K$ that measures how the zero set of $\tilde{p}$ approaches $(1,1)$ from within the face $\mathbb{T} \times \mathbb{D}$ of the bidisk. Specifically, $K$ is the number satisfying
\[ \inf \{ 1-|z_2|: \tilde{p}(z_1, z_2)=0\} \approx |1-z_1|^K\]
for $z_1\in \mathbb{T}$ and $z_2 \in \mathbb{D}$ both sufficiently close to $1$. More generally, if $\phi$ has a singularity at  $(\tau_1, \tau_2) \in \mathbb{T}^2$, then the contact order of $\phi$ at $(\tau_1, \tau_2)$ is just the contact order of $\psi(z_1, z_2):= \phi(\tau_1 z_1, \tau_2 z_2)$ at $(1,1).$

This bidisk definition translates  to a definition of contact order on the bi-upper half plane $\mathbb{H}^2$. Specifically, recalling the conformal map $\beta: \mathbb{H} \rightarrow \mathbb{D}$ given in \eqref{eqn:beta},
we can define a related polynomial
\begin{equation} \label{eqn:r}
r(w_1, w_2) = (1-iw_1)^m(1-iw_2)^n p \left( \beta(w_1), \beta(w_2)\right).
%q_2(w) &= (1-iw_1)^n(1-iw_2)^m\tilde{p} \left( \beta(w_1), \beta(w_2)\right).
\end{equation}
Then the relationship between $r$, $p$, and $\tilde{p}$ implies that $K$ is also the number satisfying 
\begin{equation} \label{eqn:COdef2} \inf \{ |\IM(w_2)|: r(x_1, w_2)=0\} \approx |x_1|^K,\end{equation}
for $x_1 \in \mathbb{R}$ $w_2 \in -\mathbb{H}$ 
 both sufficiently close to $0$. 

In this paper, the graph construction implies that the highest power of $z_2$ appearing in our rational inner functions  is at most $1$.  When $\phi$ has a singularity on $\mathbb{T}^2$, the highest power of $z_2$ in $p$ (and hence in $\phi$) cannot be $0$ and so must equal $1$. In that case, there is a particularly easy way to evaluate contact order on the bi-upper half plane.

\begin{remark} \label{rem:cor} Let $\deg p  = (n,1)$ and consider $f$ so that $\phi =\beta \circ f \circ \beta^{-1}$. Then, there is a polynomial $Q$ such that
\[
\beta \circ f = \phi \circ \beta = \frac{Q}{r},\]
where $r$ is defined in \eqref{eqn:r} and $Q$ and $r$ have no common factors. 
Clearly, $p(1,1)=0$ implies that $r(0,0)=0$.  Now find polynomials $r_1, r_2 \in \mathbb{C}[w]$ such that
\[ r(w_1, w_2) = r_1(w_1) + w_2 r_2(w_1) = r_2(w_1) \left( w_2 + \frac{r_1(w_1)}{r_2(w_1)}\right).\]
If $r_2(0)=0,$ then  since $r_1(0)=0$, $r$ would vanish along the line $w_1=0$, which cannot happen because $p$ can only have a finite number of zeros on $\mathbb{T}^2$. Thus, $r_2(0) \ne 0$, which means that we can expand the function $\frac{r_1}{r_2}$ in a power series centered at $w_1=0.$ Using \eqref{eqn:COdef2}, we can then conclude that the contact order $K$ of $\phi$ at $(1,1)$ is also the number that satisfies
\[ K = \min \left \{ k\in \mathbb{N}: \IM \Big( \left( \tfrac{r_1}{r_2}\right)^{(k)}(0)\Big) \ne 0\right \}.\]
Alternatively, let $\bar{r}_2$ denote the polynomial whose coefficients are exactly the complex complex conjugates of those of $r_2$. Then for real inputs $x$, we have $\bar{r}_2(x) = \overline{r_2(x)}$ and 
\[\frac{r_1(x)}{r_2(x)} =\frac{ r_1(x) \overline{r_2(x)}}{|r_2(x)|^2},\]
which shows that the contact order $K$ of $\phi$ at $(1,1)$ is also given by 
\[ K = \text{ the order of vanishing at $0$ of } \IM \Big( r_1(x) \overline{r_2(x)} \Big).\]

To compute this quantity in practice, we will start with writing $f=\frac{q_1}{q_2}$ for polynomials $q_1$ and $q_2$. If $q_1$ and $q_2$ have no common factors, then up to a constant factor, $r=q_2-iq_1$. However if $q_1$ and $q_2$ actually have a common factor of the form $m(z_1)$ with $m(0) \ne 0,$ then then order of vanishing of $\IM ( r_1(x) \overline{r_2(x)})$ at $0$ would be unaffected by still letting $r=q_2-iq_1$. Thus, in what follows, we simply require that $q_1$ and $q_2$ not share a common factor of $z_1$ or a factor involving $z_2.$

Additionally, in later theorems, we also encounter $\phi$ with a singularity at $(-1,1)$. As mentioned above, to compute its contact order, we computing the contact order of $\psi(z_1, z_2):= \phi(- z_1, z_2)$ at $(1,1).$ To do this computation, we note that if $\phi \circ \beta = \beta \circ f$, then $\psi \circ \beta = \beta \circ g,$
where $g(w_1, w_2) = f(-1/w_1, w_2).$ This means we can conduct the same argument as above to compute contact order, but with $g = \frac{q_1}{q_2}.$
\end{remark}

There is also a useful way to define contact order in terms of unimodular level sets of $\phi$, see Section $3$ in \cite{bps19a}, which we encode in the following remark:

\begin{remark} \label{rem:CO2} Let $\phi$ have a singularity at $(\tau_1, \tau_2) \in \mathbb{T}^2$. Then for all but a finite number of $\eta \in \mathbb{T}$, the unimodular level set 
\[ \mathcal{C}_\eta = \text{Closure of } \left\{ (z_1, z_2) \in \mathbb{T}^2 : \phi(z_1, z_2) = \eta \right\}\]
near $(\tau_1,\tau_2)$ can be parameterized by $\ell$ analytic functions $g^1_\eta, \dots, g^\ell_\eta$ (where $\ell$ does not depend on $\eta$) in the following way: there is an open neighborhood of $(\tau_1,\tau_2)$ such that all points of $\mathcal{C}_\eta$ in that set are of the form
\[(z_1, g^1_\eta(z_1)), \dots, (z_1, g^\ell_\eta(z_1))\]
 and such points are also in $\mathcal{C}_\eta$.
Then contact order is actually a measure of how these unimodular level sets agree near the singularity $(\tau_1,\tau_2)$ for different $\eta$ values. Specifically, as shown in \cite{bps19a}, for almost every $\eta_1, \eta_2 \in \mathbb{T}$ (called generic values), if we order the parameterizing functions appropriately, then
\[ K = \max_{1\le j \le \ell}\left \{ \text{Order of vanishing of } | g^j_{\eta_1}(z_1) - g^j_{\eta_2}(z_1) | \text{ at } \tau_1 \right\}.\]
In our setting, we generally have $\deg \phi = (n,1)$ for some $n$, which will imply that the number of branches of a generic $C_\eta$ will be $1$, so $\ell=1.$
\end{remark}

\subsection{Contact order and t-values} \label{subsec:COtvalues}
We can now examine the contact orders of zeros of stable polynomials $p_A$ or $p_A^t$ obtained via our construction.
We first establish the following connection between $t$ values and contact order at most related zeros of $p_A$ and $p_A^t$.

\begin{proposition} \label{prop:COt}
Let $\tau_1, \tau_2 \in \mathbb{T} \setminus \{-1\}$. By Proposition \ref{thm:8}, $p_A(\tau_1,\tau_2) = 0$ if and only if $p_A^t(\lambda_1,\lambda_2) = 0$ where 
\[ \lambda_1 = \tau_1 \ \ \text{ and } \ \ \lambda_2 = \beta\left(\dfrac{\beta^{-1}(\tau_2) - t\beta^{-1}(\tau_1)}{1 - t}\right).\]
Then the contact order of $p_A$ at $(\tau_1, \tau_2)$ equals the contact order of $p_A^t$ at $(\lambda_1, \lambda_2).$
\end{proposition}

\begin{proof} The fact that $f_A^t(z_1 ,z_2) = f_A (z_1, t z_1 +(1-t)z_2)$ and the connections between the $f_A, \phi_A$ and $f^t_A, \phi^t_A$ functions give
\[
\begin{aligned}
\phi_A^t (z_1, z_2) &= \phi_A \left( z_1, \beta\big( t \beta^{-1}(z_1) +(1-t) \beta^{-1}(z_2)\big)\right) \\
\phi_A (z_1, z_2) &= \phi_A^t \left( z_1, \beta\Big( \frac{ \beta^{-1}(z_2)-t\beta^{-1}(z_1)}{1-t}\Big)\right) 
\end{aligned}
\]
at every $(z_1,z_2) \in \overline{\mathbb{D}^2}$ with $z_1, z_2 \ne -1$. 

Assume that $p_A(\tau_1, \tau_2)=0$, so $p_A^t(\lambda_1, \lambda_2) =0$. We'll use the unimodular level set definition of contact order from Remark \ref{rem:CO2}. To that end, for each $\eta \in \mathbb{T}$, define
\[ 
\begin{aligned}
\mathcal{C}_\eta &= \text{Closure of } \left\{ (z_1, z_2) \in \mathbb{T}^2 : \phi_A(z_1, z_2) = \eta \right\}, \\
\mathcal{C}^t_\eta &= \text{Closure of } \left\{ (z_1, z_2) \in \mathbb{T}^2 : \phi_A^t(z_1, z_2) = \eta \right\}.
\end{aligned}
\]
Then for every $(z_1, z_2) \in \mathbb{T}^2$ near $(\tau_1, \tau_2)$, we have $(z_1, z_2) \in \mathcal{C}_\eta$ if and only if 
\[\left(z_1, \beta \left( \frac{\beta^{-1}(z_2)-t \beta^{-1}(z_1)}{1-t}\right)\right) \in \mathcal{C}^t_\eta.\]
Choose $\eta_1, \eta_2 \in \mathbb{T}$ generic for both $\phi_A$ and $\phi_A^t$. Because the degrees of $\phi_A$ and $\phi_A^t$ in $z_2$ are $1$, each $\mathcal{C}_{\eta_k}, \mathcal{C}^t_{\eta_k}$  can be parameterized by a single analytic function
 $z_2=g_{\eta_k}(z_1)$ or $z_2=g^t_{\eta_k}(z_1)$ and the relationship between $\mathcal{C}_\eta$ and $\mathcal{C}^t_\eta$ implies that 
\begin{equation} \label{eqn:gteta} g^t_{\eta_k}(z_1) = \beta \left( \frac{\beta^{-1}(g_{\eta_k}(z_1))-t \beta^{-1}(z_1)}{1-t}\right).  \end{equation}
 Let $K$ be the contact order of $\phi_A$ at $(\tau_1,\tau_2)$. Then  
\[
\begin{aligned}
K &= \text{Order of vanishing of } | g_{\eta_1}(z_1) - g_{\eta_2}(z_1) | \text{ at } \tau_1 \\
&= \min \left \{ L: \frac{d^L}{dz_1^L}\left(g_{\eta_1} - g_{\eta_2}\right)|_{z_1=\tau_1}\ne0\right\}\\
& =  \min \left\{L :g^{(L)}_{\eta_1}(\tau_1) \ne  g^{(L)}_{\eta_2}(\tau_1)\right\}.
\end{aligned}
\]
Let $K_t$ be the contact order of $\phi_A^t$ at $(\lambda_1,\lambda_2)$. Then  
\[ 
\begin{aligned}K_t &=  \text{Order of vanishing of } | g^{t}_{\eta_1}(z_1) - g^{t}_{\eta_2}(z_1) | \text{ at } \lambda_1  \\
& =  \min \left \{ L: \left(g^{t}_{\eta_1}\right)^{(L)}(\lambda_1) \ne  \left(g^{t}_{\eta_2}\right)^{(L)}(\lambda_1) \right\}.
\end{aligned}
\]
Using \eqref{eqn:gteta}, $\tau_1, \tau_2 \ne -1$, and $\tau_1 =\lambda_1$, we can see that if
\[g^{(\ell)}_{\eta_1}(\tau_1) =  g^{(\ell)}_{\eta_2}(\tau_1)\]
for $\ell=0, \dots, L-1$, then 
\[\left(g^{t}_{\eta_1}\right)^{(\ell)}(\lambda_1) =  \left(g^{t}_{\eta_2}\right)^{(\ell)}(\lambda_1) \]
for $\ell=0, \dots, L-1$ as well. This immediately implies that $K \le K_t.$
A similar argument (using the converse formulas) gives $K_t \le K$ and completes the proof.
\end{proof}

\subsection{Contact Order and Path Graphs}  \label{subsec:COpaths} In this subsection, we look specifically at stable polynomials constructed from path graphs and prove Theorem \ref{thm:COpaths}. Note that an $n$-vertex path graph has adjacency matrix $A_n$ as in Figure \ref{fig:pathgraph}.

\begin{center}
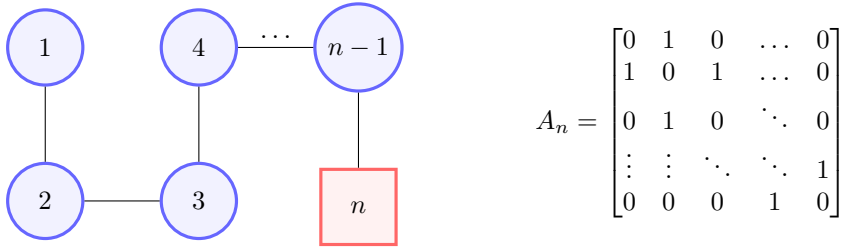
\begin{figure}[ht!]
\begin{minipage}{0.5\linewidth}
\begin{tikzpicture}[
roundnode/.style={circle, draw=blue!60, fill=blue!5, very thick, minimum size=10mm},
squarednode/.style={rectangle, draw=red!60, fill=red!5, very thick, minimum size=10mm}, 
]
\node[roundnode] (node1) {1};
\node[roundnode] (node2) [below=of node1] {2};
\node[roundnode] (node3) [right=of node2] {3};
\node[roundnode] (node4) [above=of node3] {4};
\node[roundnode] (node5) [right=of node4] {$n-1$};
\node[squarednode] (node6) [below=of node5] {$n$};

\draw[-] (node1.south) -- (node2.north);
\draw[-] (node2.east) -- (node3.west);
\draw[-] (node4.south) -- (node3.north);
\draw[-] (node4.east) -- node[midway, above] {$\hdots$} (node5.west);
\draw[-] (node5.south) -- (node6.north);
\end{tikzpicture}
\end{minipage}%
\begin{minipage}{0.4 \linewidth}
\[ A_n = \begin{bmatrix} 
0 & 1 & 0 & \hdots & 0 \\
1 & 0 & 1 & \hdots & 0 \\
0 & 1 & 0 & \ddots & 0 \\
\vdots & \vdots & \ddots & \ddots & 1 \\
0 & 0 & 0 & 1 & 0
\end{bmatrix} \].
\end{minipage}
\caption{An $n$-vertex path graph and its adjacency matrix}
\label{fig:pathgraph}
\end{figure}
\end{center}
Before proving  Theorem \ref{thm:COpaths}, we prove a related lemma. Below, for nonnegative integers $a,b$, we let ${{a} \choose {b}}$ be defined as follows:
\[{{a} \choose {b}} = \frac{a!}{b!(a-b)!} \text{ if } a \ge b \ \ \text{ and }  \ \ {{a} \choose {b}} = 0 \text{ if } a <b.\]
Then we have the following computation.

\begin{lemma} \label{lem:bionomial} Let $a_1, a_2, k \in \mathbb{N}$ with $a_1 >k$ and $a_2 \ge k$. Then $$\sum_{m=0}^{k} {{a_1 - m} \choose {m}}{{a_2 - (k-m)} \choose {k-m}} = \sum_{m=0}^{k} {{a_1 - 1 - m} \choose {m}}{{a_2 + 1 - (k-m)} \choose {k-m}}.$$
\end{lemma}

\begin{proof} For ease of notation, define 
\[
\begin{aligned} \text{Sub}(a_1,a_2,k)&:= \sum_{m=0}^{k} {{a_1 - m} \choose {m}}{{a_2 - (k-m)} \choose {k-m}} \\
&\hspace{.5in} - \sum_{m=0}^{k} {{a_1 - 1 - m} \choose {m}}{{a_2 + 1 - (k-m)} \choose {k-m}}.\end{aligned} \]
%Then, setting $b_1 =a_1- k$ and $b_2 = a_2 -k$, we can see that  $\text{Sub}(a_1,a_2,k) =0$ for all $a_1, a_2, k$ with $a_1 >k$ and $a_2 \ge k$ if and only if $\text{Sub}(b_1 + k,b_2 + k,k)=0$ for all $k, b_1, b_2 \in \mathbb{N}$ with $b_1 \ge 1.$ 
We will prove $\text{Sub}(a_1,a_2,k) =0$ for all $a_1, a_2, k$ with $a_1 >k$ and $a_2 \ge k$ by induction on $k$. For the base case $k=0$, fix any $a_1, a_2 \in \mathbb{N}$ with $a_1 >0$ and $a_2 \ge 0$ and observe that
$$\text{Sub}(a_1,a_2,0) = {{a_1 } \choose {0}}{{a_2 } \choose {0}} - {{a_1 - 1 } \choose {0}}{{a_2 + 1 } \choose {0}} = 1 \cdot 1 - 1 \cdot 1 = 0.$$

For the inductive step, assume that $\text{Sub}(c_1,c_2,k-1) =0$ for all $c_1, c_2 \in \mathbb{N}$ with $c_1 >k-1$ and $c_2 \ge k-1$. Fix $a_1 >k$ and $a_2 \ge k$. We will show  $\text{Sub}(a_1,a_2,k) =0.$ As part of that argument, we use Pascal's formula, which holds for all $j, \ell >0$ (though if $j <\ell$, then the coefficient is $0$):
\begin{equation} \label{eqn:pascal} {{j}\choose{\ell}} = {{j-1}\choose{\ell}} + {{j-1}\choose{\ell-1}}.\end{equation}
Using the formula for $\text{Sub}(a_1,a_2,k)$, extracting one term from each sum, using the Pascal's triangle formula on each sum, and canceling any common terms from the two sums gives:

\begin{align*}
\text{Sub}&(a_1,a_2,k) = \sum_{m=0}^{k} {{a_1 - m} \choose {m}}{{a_2 - (k-m)} \choose {k-m}}\\
&\hspace{.75in}- \sum_{m=0}^{k} {{a_1 - 1 - m} \choose {m}}{{a_2 + 1 - (k-m)} \choose {k-m}} \\
&= {{a_2 - k} \choose k} + \sum_{m=1}^{k} \left( {{a_1 - m - 1}\choose{m}}+ {{a_1 - m - 1}\choose{m-1}}\right) {{a_2 - (k-m)} \choose {k-m}}  \\
& \quad - \sum_{m=0}^{k-1}{{a_1 - m - 1}\choose{m}}\left({{a_2 - (k-m)} \choose {k-m}} + {{a_2 - (k-m)} \choose {k-m-1}}\right) - {{a_1-1-k}\choose{k}}  \\
&= {{a_2-k}\choose{k}} - {{a_1 - 1 - k} \choose k} + {{a_1 - 1 - k} \choose k} - {{a_2-k}\choose{k}} \\
& \quad + \sum_{m=1}^{k}{{a_1 - m - 1}\choose{m-1}}{{a_2 - (k-m)} \choose {k-m}} - \sum_{m=0}^{k-1}{{a_1 - m - 1}\choose{m}}{{a_2 - (k-m)} \choose {k-m-1}}. 
\end{align*}
Canceling terms and re-indexing the first sum gives:
\begin{align*}
\text{Sub}(a_1,a_2,k)&= \sum_{m=0}^{k-1}{{(a_1 - 1) - m - 1}\choose{m}}{{a_2 - (k-1-m)} \choose {k-1-m}}  \\
& \qquad  \qquad - \sum_{m=0}^{k-1}{{a_1 - 1-m}\choose{m}}{{a_2 -1- (k-1-m)} \choose {k-1-m}} \\
&= -\text{Sub}(a_1-1,a_2-1,k-1 ) =0,\end{align*}
by the inductive hypothesis. This completes the proof.
\end{proof}

Now we can prove Theorem \ref{thm:COpaths}:

\begin{proof} Since $v_1$ and $v_n$ are path connected, Theorem $\ref{thm:1}$ implies that $p_{A_n}(-1,1)=0.$ We will use Remark \ref{rem:cor} to show that the contact order of $p_A$ at $(-1,1)$ is $2n-2$. First write
\[ g_{A_n}(z_1, z_2) : = f_{A_n} \left(-\tfrac{1}{z_1}, z_2 \right) = \frac{q_1}{q_2}(z_1, z_2),  \]
where the polynomials $q_1, q_2$ have no common factors of $z_1$ or common factors involving $z_2$ and define $r_1, r_2 \in \mathbb{C}[z]$ so that 
\begin{equation} \label{eqn:r1r2} \left( q_2- i q_1\right) (z_1, z_2) 
= r_1 (z_1) + z_2 r_2(z_1). 
\end{equation}
Then the contact order of $p_{A_n}$ at $(-1,1)$ is the order of vanishing of $\IM(r_1(x) \overline{r_2(x)})$ near $0$. To compute that, let $\widehat{z}_Y$ be the diagonal matrix (of appropriate size) with entries $-\frac{1}{z_1}, -\frac{1}{z_1}, \dots, -\frac{1}{z_1}, z_2$ along the diagonal and observe that for $n >3$,
\begin{equation} \label{eqn:gan}
g_{A_n}(z_1,z_2) = \frac{\det(A_{n-1} - \widehat{z}_Y)}{\det(A_{n} - \widehat{z}_Y)} =  \frac{-z_2\det(A_{n-2} +\frac{1}{z_1}I)-\det(A_{n-3} +\frac{1}{z_1}I) }{-z_2\det(A_{n-1} +\frac{1}{z_1}I)-\det(A_{n-2} +\frac{1}{z_1}I) },\end{equation}
where we used the structure of the $A_n$ matrices. This shows that we will need a formula for
\[ G(n, z_1):=\det(A_n - z_1 I).\]
Direct computation gives $G(1, z_1)=-z_1,$ $G(2, z_1)= z_1^2 - 1$, and the structure of the $A_n$ matrices gives the recursive formula 
\[ G(n,z_1) = -z_1 G(n-1, z_1) - G(n-2, z_1).\]
%$det(A_n - z_1 I) = -z_1 det(A_{n-1}) - det (A_{n-2}), \quad \text{ for } n \ge 3.$
We claim that 
\[ G(n, z_1)=
\begin{cases}\displaystyle
\sum_{m=0}^{\frac{n-1}{2}} {{\frac{n+1}{2}+m}\choose{\frac{n-1}{2}-m}} (-1)^{\frac{n+1}{2}-m}z_1^{2m+1}, & \textrm{for $n$ odd},\\
\displaystyle \sum_{m=0}^{\frac{n}{2}} {{\frac{n}{2}+m}\choose{\frac{n}{2}-m}} (-1)^{\frac{n}{2}-m}z_1^{2m}, & \textrm{ for $n$ even.}
\end{cases}\]
Our earlier formulas imply that this holds for $n=1,2$ and we prove the result via induction for $n \ge 3$. Specifically, we assume that the formula holds for $n=k-2$ and $n=k-1$ and we show that it holds for $n=k$.  Since the cases are almost identical, we show the case when $k$ is odd and leave the case when $k$ is even to the reader. 

Since $k$ is odd, we know $k-1$ is even and $k-2$ is odd. Applying the inductive hypothesis, extracting the last term of the first sum, using Pascal's formula \eqref{eqn:pascal} to combine the two sums, and then adding the final term back in gives
\begin{align*}
G(k, z_1) &= -z_1 G(k-1, z_1) - G(k-2, z_1) \\ 
&= - \sum_{m=0}^{\frac{k-1}{2}} {{\frac{k-1}{2} + m}\choose{\frac{k-1}{2} - m}} (-1)^{\frac{k-1}{2}-m} z_1^{2m+1} 
-\sum_{m=0}^{\frac{k-3}{2}} {{\frac{k-1}{2} + m}\choose{\frac{k-3}{2} - m}} (-1)^{\frac{k-1}{2}-m} z_1^{2m+1}\\
&= \sum_{m=0}^{\frac{k-3}{2}} {{\frac{k+1}{2} + m}\choose{\frac{k-1}{2} - m}} (-1)^{\frac{k+1}{2}-m} z_1^{2m+1} - z_1^{k}\\
&= \sum_{m=0}^{\frac{k-1}{2}} {{\frac{k+1}{2} + m}\choose{\frac{k-1}{2} - m}} (-1)^{\frac{k+1}{2}-m} z_1^{2m+1},
\end{align*}
as needed. 

To study $g_{A_n}$, we consider
\[
\begin{aligned} h(n,z_1):= z_1^n G\left (n,\tfrac{1}{z_1}\right) &=\begin{cases} \displaystyle
\sum_{m=0}^{\frac{n-1}{2}} {{\frac{n+1}{2}+m}\choose{\frac{n-1}{2}-m}} (-1)^{\frac{n+1}{2}-m}z_1^{n - 2m - 1}, & \textrm{for $n$ odd},\\
\displaystyle \sum_{m=0}^{\frac{n}{2}} {{\frac{n}{2}+m}\choose{\frac{n}{2}-m}} (-1)^{\frac{n}{2}-m}z_1^{n - 2m}, & \textrm{for $n$ even}
\end{cases}\\
& =\begin{cases} \displaystyle
\sum_{m=0}^{\frac{n-1}{2}} {{n - m}\choose{m }} (-1)^{m+1 }z_1^{2m}, & \textrm{for $n$ odd,}\\
 \displaystyle \sum_{m=0}^{\frac{n}{2}} {{n - m}\choose{m}} (-1)^{m}z_1^{2m} & \textrm{for $n$ even,}
\end{cases}
\end{aligned}\]
where the second line following by re-indexing the sums (basically replacing $m$ with $\frac{n-1}{2} - m$ for odd cases, and $m$ with $\frac{n}{2} - m$ for even cases). Then from \eqref{eqn:gan}
\begin{align*}
g_{A_n}(z_1,z_2) &= \frac{-z_2G(n-2,-\frac{1}{z_1}) - G(n-3,-\frac{1}{z_1})}{-z_2G(n-1,-\frac{1}{z_1}) - G(n-2,-\frac{1}{z_1})} \\
&= \frac{z_1^{n-1}\big(z_2G(n-2,-\frac{1}{z_1}) + G(n-3,-\frac{1}{z_1})\big)}{z_1^{n-1}\big(z_2G(n-1,-\frac{1}{z_1}) + G(n-2,-\frac{1}{z_1})\big)} \\
&= \frac{z_1 z_2h(n-2,-z_1) + z^2_1 h(n-3,-z_1)}{z_2h(n-1,-z_1) + z_1 h(n-2,-z_1)}\\
&=\frac{z_1z_2 h(n-2,z_1) + z^2_1 h(n-3,z_1)}{z_2h(n-1,z_1) +z_1 h(n-2,z_1)},
\end{align*}
since the powers of $z_1$ in the formula for $h(n,z_1)$ are always even. The numerator and denominator of the above fraction give us $q_1$ and $q_2$; specifically, we can see that they have no common factor involving $z_2$ since $p_{A_n}$ has a zero at $(-1,1)$ and no common factor of $z_1$ since $h(n-1,z_1)$ has a non-zero constant term. Then, we can compute
\[ 
\begin{aligned}\left( q_2- i q_1\right)& (z_1, z_2) 
= z_2 h(n-1,z_1) +z_1 h(n-2,z_1) 
\\ & \hspace{1in} - i\Big( z_1z_2 h(n-2,z_1)  + z^2_1 h(n-3,z_1)\Big) \\
&=\big( z_1 h(n-2,z_1)- iz^2_1 h(n-3,z_1) \big)+z_2\big( h(n-1,z_1) - iz_1 h(n-2,z_1)\big).
\end{aligned}
\]
This gives us $r_1$ and $r_2$ via \eqref{eqn:r1r2} and so, we just need to compute the order of vanishing of 
\[\IM(r_1(x) \overline{r_2(x)}) = x^2 h(n-2,x)h(n-2,x) -x^2h(n-1,x)h(n-3,x),
\]
Note that this requires $n >3$. The cases $n=2,3$ can be easily checked by hand.
As before, we assume that $n$ is odd, since the even case can be handled very similarly.  Then
\begin{align*}
-\IM(r_1(x) \overline{r}_2(x))  
&= x^2\sum_{m=0}^{\frac{n-1}{2}} {{n-1-m}\choose{m}} (-1)^mx^{2m} \sum_{j=0}^{\frac{n-3}{2}} {{n-3-j}\choose{j}} (-1)^jx^{2j} \\
&- x^2 \sum_{m=0}^{\frac{n-3}{2}} {{n-2-m}\choose{m}} (-1)^{m+1} x^{2m} \sum_{j=0}^{\frac{n-3}{2}} {{n-2-j}\choose{j}} (-1)^{j+1}x^{2j} \\
&= x^2\sum_{k=0}^{n-2} c_k x^{2k},
\end{align*}
for coefficients $c_k \in \mathbb{R}$. We claim that all coefficients are $0$, except $c_{n-2}$. First note that
\[c_{n-2} = {{\tfrac{n-1}{2}}\choose{\frac{n-1}{2}}}{{\tfrac{n-3}{2}}\choose{\frac{n-3}{2}}}(-1)^{n-2}  = (-1)^{n-2} = -1. \]
To show that the other coefficients are zero, we consider two cases. First, if $0 \le k \le \frac{n-3}{2},$ then $c_k$ equals
\[ (-1)^{k}\sum_{m=0}^{k} \left( {{n-1-m}\choose{m}}  {{n-3-(k-m)}\choose{k-m}}-
{{n-2-m}\choose{m}}  {{n-2-(k-m)}\choose{k-m}}\right), \]
which is $0$ by Lemma \ref{lem:bionomial}. Meanwhile, if $\frac{n-1}{2} \le k \le n-3$, then we have some restrictions on the terms that appear in the formula for $c_k.$ Namely, one can check that $c_k$ equals $(-1)^k$ times
\[\sum_{m=k-\frac{n-3}{2}}^{\frac{n-1}{2}} \hspace{-.08in} {{n-1-m}\choose{m}}  {{n-3-(k-m)}\choose{k-m}}-\hspace{-.08in} \sum_{m=k-\frac{n-3}{2}}^{\frac{n-3}{2}} \hspace{-.08in} {{n-2-m}\choose{m}}  {{n-2-(k-m)}\choose{k-m}}. \]
However, when $m$ is both outside of the current sum bounds and satisfies $0\le m \le k$, the binomial coefficient formulas in the summands are both well defined and equal zero.  Thus, we can also write the formula for $c_k$ as 
\[ (-1)^{k}\sum_{m=0}^{k} \left( {{n-1-m}\choose{m}}  {{n-3-(k-m)}\choose{k-m}}-
{{n-2-m}\choose{m}}  {{n-2-(k-m)}\choose{k-m}}\right), \]
which is again $0$ by Lemma \ref{lem:bionomial}. This means that for $n$ odd, 
\[\IM(r_1(x) \overline{r_2(x)}) = x^{2n-2}\]
and so the contact order of $p_{A_n}$ at $(-1,1)$ is $2n-2$. An analogous argument gives the result for $n$ even.
\end{proof}

\subsection{Contact Order with Graph Modifications} \label{subsec:GM}

In this section, we investigate Conjecture \ref{con:CO} and prove Theorem \ref{thm:CO2}, which gives three cases that support the conjecture. We split the proof into the three different cases below.

Throughout these proofs we will use the following notation: for an arbitrary $n\times n$ matrix $M$, we let $\overline{M}$ denote the $(n-1) \times (n-1)$ matrix obtained from $M$ by removing its first row and first column, $\underline{M}$ 
denote the $(n-1) \times (n-1)$ matrix obtained from $M$ by removing its last row and last column, and $\overline{\underline{M}}$ denote the $(n-2) \times (n-2)$ matrix obtained from $M$ by removing both its first row and first column and its last row and last column.

For the first case in Theorem \ref{thm:CO2}, let $G_A$ be an $n$-vertex graph (as always, undirected and simple) with the vertices $v_1$ and $v_n$ path connected. We define a new graph $G_{\widehat{A}}$ as the graph we get from $G_A$ by adding a vertex $v_{n+1}$ that is only connected directly to $v_n.$ Specifically, $G_{\widehat{A}}$ has the adjacency matrix  
\begin{equation}\label{eqn:A_hat}
    \widehat{A} = \left[
\begin{array}{c|c}
A & \begin{array}{c} 0 \\ \vdots \\ 0 \\ 1 \end{array} \\
\hline
\begin{array}{ccccc} 0 & \cdots & 0 & 1\end{array} & 0
\end{array}
\right].
\end{equation}
An example of this construction is given in Figure \ref{fig:extendtail} below. 
\begin{figure}[ht!] 
\begin{minipage}{0.30\linewidth}
\begin{center}
\begin{tikzpicture}[
roundnode/.style={circle, draw=blue!60, fill=blue!5, very thick, minimum size=7mm},
squarednode/.style={rectangle, draw=red!60, fill=red!5, very thick, minimum size=7mm}, 
]
\node[roundnode]      (node1)                              {1};
\node[roundnode]        (node2)       [below=of node1] {2};
  \node[roundnode]    (node3)       [right=of node2] {3};
 \node[squarednode]       (node4)       [right=of node1] {4};
\draw[-] (node1.south) -- (node2.north);
\draw[-] (node2.east) -- (node3.west);
\draw[-]  (node4.west) -- (node1.east);
\draw[-]  (node4.south) -- (node3.north);
\end{tikzpicture}
\end{center}
\end{minipage}%
\begin{minipage}{0.03\linewidth}
$\rightarrow$
\end{minipage}
\begin{minipage}{0.35\linewidth}
\begin{center}
\begin{tikzpicture}[
roundnode/.style={circle, draw=blue!60, fill=blue!5, very thick, minimum size=7mm},
squarednode/.style={rectangle, draw=red!60, fill=red!5, very thick, minimum size=7mm}, 
]
\node[roundnode]      (node1)                              {1};
\node[roundnode]        (node2)       [below=of node1] {2};
  \node[roundnode]    (node3)       [right=of node2] {3};
 \node[roundnode]       (node4)       [right=of node1] {4};
  \node[squarednode]       (node5)       [right=of node4] {5};
\draw[-] (node1.south) -- (node2.north);
\draw[-] (node2.east) -- (node3.west);
\draw[-]  (node4.west) -- (node1.east);
\draw[-]  (node5.west) -- (node4.east);
\draw[-]  (node4.south) -- (node3.north);
\end{tikzpicture}
\end{center}
\end{minipage}%

\caption{An example of the graph modification $G_A$ to $G_{\widehat{A}}$ for Theorem \ref{thmCO:i}}
\label{fig:extendtail}
\end{figure}

Then we have the following result.

\begin{theorem} \label{thmCO:i} Let $G_A$ be an $n$-vertex graph with $v_1$ and $v_n$ path connected. Assume that the contact order of $p_A$ at $(-1,1)$ is $K$. 
Then, with $\widehat{A}$ defined in \eqref{eqn:A_hat}, the contact order of $p_{\widehat{A}}$ at $(-1,1)$ is $K+2$.
\end{theorem}

\begin{proof} We will use Remark \ref{rem:cor} to obtain tractable formulas for the contact orders of $p_A$ and $p_{\widehat{A}}$ at $(-1,1)$ that will allow us to compare the two quantities directly. To that end, we consider polynomials $q_1, q_2, \hat{q}_1, \hat{q_2}$ that satisfy
\[ g_A(z_1, z_2) : = f_A \left(-\tfrac{1}{z_1}, z_2 \right) = \frac{q_1}{q_2} \ \ \text{ and } \ \ g_{\widehat{A}}(z_1, z_2) : = f_{\widehat{A}}\left(-\tfrac{1}{z_1}, z_2 \right) = \frac{\hat{q}_1}{\hat{q}_2}  \]
and the pairs $(q_1, q_2)$ and $(\hat{q}_1, \hat{q}_2)$ do not have any common factors of $z_1$ or common factors involving $z_2.$
Then if we define $r_1, r_2, \hat{r}_1, \hat{r}_2 \in \mathbb{C}[z]$ so that they satisfy 
\begin{equation} \label{eqn:qrconn} \left( q_2- i q_1\right) (z_1, z_2) 
= r_1 (z_1) + z_2 r_2(z_1) 
\ \ \text{ and } 
\ \ \left( \hat{q}_2- i \hat{q}_1\right) (z_1, z_2) = \hat{r}_1 (z_1) + z_2 \hat{r}_2(z_1), \end{equation}
$K$ is the order of vanishing of $\IM( r_1(x) \overline{r_2(x)})$ near $0$ and the contact order of $p_{\widehat{A}}$ at $(-1,1)$ is the order of vanishing of $\IM( \hat{r}_1(x) \overline{\hat{r}_2(x)})$ near $0$. 

Now we find formulas for $q_1, q_2, \hat{q}_1, \hat{q_2}$. Let $\widehat{z}_Y$ be the diagonal matrix (of appropriate size) with entries $-\frac{1}{z_1}, -\frac{1}{z_1}, \dots, -\frac{1}{z_1}, z_2$ along the diagonal. Then
\begin{equation}\label{eqn:gAcok} g_A(z_1, z_2) = \frac{\det(\overline{A} - \widehat{z}_Y)}{\det(A - \widehat{z}_Y)} = \frac{ z_1^{n-1} \left( -z_2 \det\left(\overline{\underline{A}} + \tfrac{1}{z_1}I\right) -s_1\left( \tfrac{1}{z_1}\right) \right) }{ z_1^{n-1} \left( -z_2 \det\left(\underline{A} + \tfrac{1}{z_1}I\right) -s_2\left( \tfrac{1}{z_1}\right) \right) },
\end{equation}
where we expanded the determinant using the Laplace expansion method along the last column and last row of $(A- \widehat{z}_Y)$ and $(\overline{A}- \widehat{z}_Y)$ and where $s_1, s_2 \in \mathbb{R}[z]$ satisfy $\deg s_1 \le n-3$ and $\deg s_2 \le n-2$. Then we can take the numerator and denominator of \eqref{eqn:gAcok} to be $q_1, q_2$, since the denominator does not vanish identically when $z_1=0$ (so there is no common factor of $z_1$) and vertices $v_1$ and $v_n$ are path connected (so a factor involving $z_2$ cannot cancel). 

For $g_{\widehat{A}}$, we can use Laplace expansion first along the last row and column to reduce to a formula involving $A$ and then apply Laplace expansion again (in a way analogous to that for $g_A$) to obtain
\[ \begin{aligned} 
  g_{\widehat{A}}(z_1, z_2) &= \frac{\det\left (\overline{\widehat{A}} - \widehat{z}_Y\right)}{\det\left(\widehat{A} - \widehat{z}_Y\right)} = \frac{ -z_2 \det \left(\overline{A} + \tfrac{1}{z_1} I\right) - \det\left(\overline{\underline{A}} +\tfrac{1}{z_1}I\right)}{ -z_2 \det \left(A + \tfrac{1}{z_1} I\right) - \det\left(\underline{A} +\tfrac{1}{z_1} I\right)} \\
&= \frac{ z_1^n\left( z_2 \left( {-} \frac{1}{z_1} \det \left(\overline{\underline{A}} + \tfrac{1}{z_1} I\right) + s_1\left( \frac{1}{z_1} \right) \right) - \det\left(\overline{\underline{A}} +\tfrac{1}{z_1}I\right)\right)}{ z_1^n\left(z_2 \left( {-} \frac{1}{z_1}\det \left(\underline{A} + \tfrac{1}{z_1} I\right) +s_2\left(\frac{1}{z_1}\right) \right) -\det\left(\underline{A} +\tfrac{1}{z_1}I\right)\right)}. \label{eqn:gAcok2}
\end{aligned} \]
As before, we can take the numerator and denominator of the above expression to be $\hat{q}_1, \hat{q}_2$. From this and \eqref{eqn:qrconn}, we can easily identify $r_1, r_2, \hat{r}_1,$ and $\hat{r}_2:$
\[
\begin{aligned}
r_1(z_1)&= -z_1^{n-1} \left( s_2\left( \tfrac{1}{z_1}\right) - i    s_1\left( \tfrac{1}{z_1}\right) \right) \\
r_2(z_1) &=  -z_1^{n-1} \left( \det\left(\underline{A} + \tfrac{1}{z_1}I\right)- i \det\left(\overline{\underline{A}} + \tfrac{1}{z_1}I\right) \right) \\
\hat{r}_1(z_1)&=-z_1^{n}\left( \det\left(\underline{A} +\tfrac{1}{z_1}I\right) - i \det\left(\overline{\underline{A}} +\tfrac{1}{z_1}I\right)\right) \\
\hat{r}_2(z_1)& =z_1^{n}\left( {-} \tfrac{1}{z_1}\det \left(\underline{A} + \tfrac{1}{z_1} I\right) +s_2\left(\tfrac{1}{z_1}\right) - i \left({-} \tfrac{1}{z_1} \det \left(\overline{\underline{A}} + \tfrac{1}{z_1} I\right) + s_1\left( \tfrac{1}{z_1} \right) \right) \right).
\end{aligned}
\]
By assumption, we know for $x \in \mathbb{R}$ near $0$,
\begin{equation} \label{eqn:coks1s2} x^K \approx \IM( r_1(x) \overline{r_2(x)})  = x^{2n-2} \left(s_2\left( \tfrac{1}{x}\right)\det\left(\overline{\underline{A}} + \tfrac{1}{x}I\right) - s_1\left( \tfrac{1}{x}\right) \det\left(\underline{A} + \tfrac{1}{x}I\right)\right).  \end{equation}
Then to compute the contact order of $p_{\widehat{A}}$ at $(-1,1)$, just observe that for $x\in \mathbb{R}$ near $0$,
\[
\begin{aligned}
\IM( \hat{r}_1(x) \overline{\hat{r}_2(x)}) &= -x^{2n} \Big( \det\left(\underline{A} +\tfrac{1}{x}I\right)\left(- \tfrac{1}{x} \det \left(\overline{\underline{A}} + \tfrac{1}{x} I\right) + s_1\left( \tfrac{1}{x} \right) \right) \\
& \qquad -\det\left(\overline{\underline{A}} +\tfrac{1}{x}I\right)\left( -\tfrac{1}{x}\det \left(\underline{A} + \tfrac{1}{x} I\right) +s_2\left(\tfrac{1}{x}\right) \right) \Big) \\
&= x^2x^{2n-2} \left(s_2\left( \tfrac{1}{x}\right)\det\left(\overline{\underline{A}} + \tfrac{1}{x}I\right) - s_1\left( \tfrac{1}{x}\right) \det\left(\underline{A} + \tfrac{1}{x}I\right)\right) \approx x^{2+K},
\end{aligned}
\]
by \eqref{eqn:coks1s2}. \end{proof}

    For the second case in Theorem \ref{thm:CO2}, let $G_A$ be an $n$-vertex graph with $v_1$ and $v_n$ path connected.  Create $G_{\widetilde{A}}$ from $G_A$ by adding a vertex $v_{0}$ that is only connected directly to $v_1$ and then renumbering the vertices to $v_1,\dots, v_{n+1}$. Then $G_{\widetilde{A}}$ has the adjacency matrix  
    \begin{equation}\label{eqn:A_tilde}
        \widetilde{A} = \left[
\begin{array}{c|c}
0 & \begin{array}{cccc} 1 & 0 &\cdots & 0 \end{array} \\
\hline
\begin{array}{c}1 \\ 0 \\ \vdots \\ 0 \end{array} & A
\end{array}
\right].
    \end{equation}
This construction is illustrated in Figure \ref{fig:extendhead} below.

\begin{figure}[ht!] 
\begin{minipage}{0.30\linewidth}
\begin{center}
\begin{tikzpicture}[
roundnode/.style={circle, draw=blue!60, fill=blue!5, very thick, minimum size=7mm},
squarednode/.style={rectangle, draw=red!60, fill=red!5, very thick, minimum size=7mm}, 
]
\node[roundnode]      (node1)                              {1};
\node[roundnode]        (node2)       [below=of node1] {2};
  \node[roundnode]    (node3)       [right=of node2] {3};
 \node[squarednode]       (node4)       [right=of node1] {4};
\draw[-] (node1.south) -- (node2.north);
\draw[-] (node2.east) -- (node3.west);
\draw[-]  (node4.west) -- (node1.east);
\draw[-]  (node4.south) -- (node3.north);
\end{tikzpicture}
\end{center}
\end{minipage}%
\begin{minipage}{0.05\linewidth}
$\rightarrow$
\end{minipage}
\begin{minipage}{0.35\linewidth}
\begin{center}
\begin{tikzpicture}[
roundnode/.style={circle, draw=blue!60, fill=blue!5, very thick, minimum size=7mm},
squarednode/.style={rectangle, draw=red!60, fill=red!5, very thick, minimum size=7mm}, 
]
\node[roundnode]      (node1) {1}; 
\node[roundnode]      (node2) [right=of node1]{2};
\node[roundnode]        (node3)       [below=of node2] {3};
  \node[roundnode]    (node4)       [right=of node3] {4};
 \node[squarednode]       (node5)       [right=of node2] {5};
\draw[-]  (node2.west) -- (node1.east);
\draw[-] (node2.south) -- (node3.north);
\draw[-] (node3.east) -- (node4.west);
\draw[-]  (node5.west) -- (node2.east);
\draw[-]  (node5.south) -- (node4.north);
\end{tikzpicture}
\end{center}
\end{minipage}%
\caption{An example of the graph modification $G_A$ to $G_{\widetilde{A}}$ for Theorem  \ref{thmCO:ii}}
\label{fig:extendhead}
\end{figure}
    
We then have the following result.

\begin{theorem} \label{thmCO:ii} Let $G_A$ be an $n$-vertex graph with vertices $1$ and $n$ path connected. Assume that the contact order of $p_A$ at $(-1,1)$ is $K$. Then, with $\widetilde{A}$ defined in \eqref{eqn:A_tilde}, the contact order of $p_{\widetilde{A}}$ at $(-1,1)$ is $K+2$.
\end{theorem}

\begin{proof} We proceed as in the proof of Theorem \ref{thmCO:i}. Then from that  proof, we can extract the following useful identities from \eqref{eqn:gAcok} and \eqref{eqn:coks1s2}:
\begin{eqnarray}
\det(\overline{A} - \widehat{z}_Y) &=& -z_2 \det\left(\overline{\underline{A}} + \tfrac{1}{z_1}I\right) -s_1\left( \tfrac{1}{z_1}\right) \label{eqn:detco1} \\
\det(A - \widehat{z}_Y) &=& -z_2 \det\left(\underline{A} + \tfrac{1}{z_1}I\right) -s_2\left( \tfrac{1}{z_1}\right) \label{eqn:detco2} \\
x^K &\approx&   x^{2n-2} \left(s_2\left( \tfrac{1}{x}\right)\det\left(\overline{\underline{A}} + \tfrac{1}{x}I\right) - s_1\left( \tfrac{1}{x}\right) \det\left(\underline{A} + \tfrac{1}{x}I\right)\right), \label{eqn:firstco} 
\end{eqnarray}
for $x \in \mathbb{R}$ near $0$,
where $K$ is the contact order of $p_A$ at $(-1,1)$ and $s_1, s_2, \widehat{z}_Y$ are defined in the proof of Theorem \ref{thmCO:i}.
As in the proof of Theorem \ref{thmCO:i}, we now find polynomials $\tilde{q}_1, \tilde{q}_2$ (with no common factor of $z_1$ or common factor containing $z_2$) such that 
\[g_{\widetilde{A}}(z_1, z_2) : = f_{\widetilde{A}}\left(-\tfrac{1}{z_1}, z_2 \right) = \frac{\tilde{q}_1}{\tilde{q}_2},  \]
write $\tilde{q}_2 -i\tilde{q}_1 = \tilde{r}_1(z_1) + z_2 \tilde{r}_2(z_1)$ for polynomials $\tilde{r}_1, \tilde{r}_2 \in \mathbb{C}[z]$, and compute the order of vanishing of $\IM( \tilde{r}_1(x) \overline{\tilde{r}_2(x)})$ for $x\in \mathbb{R}$ near $0$. 

To that end, observe that using Laplace expansion along the first row and column of $(\widetilde{A}- \widehat{z}_Y)$, $(\overline{\widetilde{A}}- \widehat{z}_Y)$ and then using the formulas \eqref{eqn:detco1} and \eqref{eqn:detco2} will give

\begin{align}
g_{\widetilde{A}}(z_1, z_2) &= \frac{\det\left (\overline{\widetilde{A}} - \widehat{z}_Y\right)}{\det\left(\widetilde{A} - \widehat{z}_Y\right)}
=\frac{\det\left (A- \widehat{z}_Y\right)}{\frac{1}{z_1}\det(A-\widehat{z}_Y) - \det\left(\overline{A}-\widehat{z}_Y\right)} \nonumber\\
&= \frac{-z_2 \det\left(\underline{A} + \tfrac{1}{z_1}I\right) -s_2\left( \tfrac{1}{z_1}\right)}{\frac{1}{z_1}\left(-z_2 \det\left(\underline{A} + \tfrac{1}{z_1}I\right) -s_2\left( \tfrac{1}{z_1}\right)\right) -\left(-z_2 \det\left(\overline{\underline{A}} + \tfrac{1}{z_1}I\right) -s_1\left( \tfrac{1}{z_1}\right) \right)} \nonumber\\
&= \frac{z_1^{n}\left(z_2 \det\left(\underline{A} + \tfrac{1}{z_1}I\right) +s_2\left( \tfrac{1}{z_1}\right)\right)}{z_1^{n}\left(z_2 \left(\frac{1}{z_1}\det\left(\underline{A} + \tfrac{1}{z_1}I\right) -\det\left(\overline{\underline{A}} + \tfrac{1}{z_1}I\right)\right) +
\tfrac{1}{z_1} s_2\left( \tfrac{1}{z_1}\right) - s_1\left( \tfrac{1}{z_1}\right) \right)}. \label{eqn:q1q2final}
\end{align}
We can let the numerator and denominator in \eqref{eqn:q1q2final} be $\tilde{q}_1, \tilde{q}_2$ (there is no common $z_1$ factor since the denominator has a nonzero constant term and no term with $z_2$ can cancel since vertices $v_1$ and $v_{n+1}$ are path connected) and easily find $\tilde{r}_1$ and $\tilde{r}_2$:
\[
\begin{aligned}
\tilde{r}_1(z_1)&= z_1^{n} \left( \tfrac{1}{z_1} s_2\left( \tfrac{1}{z_1}\right) - s_1\left( \tfrac{1}{z_1}\right) -i s_2\left( \tfrac{1}{z_1}\right) \right)\\
\tilde{r}_2(z_1)&=z_1^{n} \left(\tfrac{1}{z_1}\det\left(\underline{A} + \tfrac{1}{z_1}I\right) -\det\left(\overline{\underline{A}} + \tfrac{1}{z_1}I\right)- i\det\left(\underline{A} + \tfrac{1}{z_1}I\right) \right). 
\end{aligned}
\]
Now we can compute the contact order for $p_{\widetilde{A}}$ at $(-1,1)$. Specifically, for $x\in \mathbb{R}$ near $0$,
\[ 
\begin{aligned}
\IM( \tilde{r}_1(x) \overline{\tilde{r}_2(x)}) &= x^{2n} \Big( 
 \left( \tfrac{1}{x} s_2\left( \tfrac{1}{x}\right) - s_1\left( \tfrac{1}{x}\right)\right) \det\left(\underline{A} + \tfrac{1}{x}I\right) 
\\
& \qquad -s_2\left( \tfrac{1}{x}\right) \left(\tfrac{1}{x}\det\left(\underline{A} + \tfrac{1}{x}I\right) -\det\left(\overline{\underline{A}} + \tfrac{1}{x}I\right)\right)  \Big) \\
&= x^2 x^{2n-2} \left(s_2\left( \tfrac{1}{x}\right)\det\left(\overline{\underline{A}} + \tfrac{1}{x}I\right) - s_1\left( \tfrac{1}{x}\right) \det\left(\underline{A} + \tfrac{1}{x}I\right)\right) \approx x^{2+K},
\end{aligned}
\]
where we used \eqref{eqn:firstco}.
Then applying Remark \ref{rem:cor} establishes the claim.
\end{proof}

Let $G_A$ be an $n$-vertex graph such that the only connection between vertex $n$ and the rest of the graph is a single edge between vertex $n-1$ and vertex $n$. Then $G_{A}$ has the adjacency matrix
$$A = \left[
\begin{array}{c|c}
B & \begin{array}{c} 0 \\ \vdots \\ 0 \\ 1 \end{array} \\
\hline
\begin{array}{cccc} 0 & \cdots & 0 & 1\end{array} & 0
\end{array}
\right], $$ for some $(n-1) \times (n-1)$ matrix $B$.

Now, fix $m \ge n+1$.  Relabel vertex $n$ to vertex $m$ and
expand $G_{A}$ to a new graph $G_{\breve{A}}$ by attaching an $(m-n)$-vertex graph to $G_A$ via an edge at vertex $n-1$ (and potentially other edges to the new vertex $m$). This follows Modification (iii) in Theorem \ref{thm:CO2}. Then $G_{\breve{A}}$ has the $m\times m$ adjacency matrix
\begin{equation}\label{eqn:A_breve}
    \breve{A} = \left[\begin{array}{c|c}
B & \begin{array}{cccc} 0 & \cdots & \cdots & 0 \\ \vdots & \ddots & \ddots & \vdots \\ 0 & \cdots & \cdots & 0 \\ * & \cdots & * & 1 \end{array} \\
\hline
\begin{array}{cccc}0 & \cdots & 0 & * \\ \vdots & \ddots & \vdots & \vdots\\ \vdots & \ddots & \vdots & *\\ 0 & \cdots & 0 & 1 \end{array} & C
\end{array}\right],
\end{equation}
 where the asterisks * contain at least one $1$ and are $0$ otherwise, and $C$ is the adjacency matrix of an $(m-(n-1))$-vertex graph $G_C$. This construction for $n=4$ and $m=6$ is illustrated in Figure \ref{fig:attachxgraph} below. The dotted edge indicates that one could also insert an edge between $v_4$ and $v_6$, and the contact order result would still hold.

\begin{figure}[ht!] 
\begin{minipage}{0.27\linewidth}
\begin{center}
\begin{tikzpicture}[
roundnode/.style={circle, draw=blue!60, fill=blue!5, very thick, minimum size=7mm},
squarednode/.style={rectangle, draw=red!60, fill=red!5, very thick, minimum size=7mm}, 
]
\node[roundnode]      (node1)                              {1};
\node[roundnode]        (node2)       [below=of node1] {2};
  \node[roundnode]    (node3)       [right=of node2] {3};
 \node[squarednode]       (node4)       [right=of node1] {4};
\draw[-] (node1.south) -- (node2.north);
\draw[-] (node2.east) -- (node3.west);
\draw[-]  (node4.south) -- (node3.north);
\end{tikzpicture}
\end{center}
\end{minipage}%
\begin{minipage}{0.04\linewidth}
$+$
\end{minipage}
\begin{minipage}{0.14\linewidth}
\begin{center}
\begin{tikzpicture}[
roundnode/.style={circle, draw=blue!60, fill=blue!5, very thick, minimum size=7mm},
squarednode/.style={rectangle, draw=red!60, fill=red!5, very thick, minimum size=7mm}, 
]
\node[roundnode]      (node1)                              {1};
\node[roundnode]        (node2)       [below=of node1] {2};

\draw[-] (node1.south) -- (node2.north);
\end{tikzpicture}
\end{center}
\end{minipage}%
\begin{minipage}{0.04\linewidth}
$\rightarrow$
\end{minipage}
\begin{minipage}{0.35\linewidth}
\begin{center}
\begin{tikzpicture}[
roundnode/.style={circle, draw=blue!60, fill=blue!5, very thick, minimum size=7mm},
squarednode/.style={rectangle, draw=red!60, fill=red!5, very thick, minimum size=7mm}, 
]
\node[roundnode]      (node1)                              {1};
\node[roundnode]        (node2)       [below=of node1] {2};
  \node[roundnode]    (node3)       [right=of node2] {3};
 \node[squarednode]       (node6)       [right=of node1] {6};
   \node[roundnode]       (node4)       [right=of node6] {4};
    \node[roundnode]       (node5)       [below=of node4] {5};
\draw[-] (node1.south) -- (node2.north);
\draw[-] (node2.east) -- (node3.west);
\draw[-]  (node3.north) -- (node6.south);
\draw[-]  (node4.south) -- (node5.north);
\draw[-]  (node5.west) -- (node3.east);
\draw[dashed]  (node4.west) -- (node6.east);
\end{tikzpicture}
\end{center}
\end{minipage}%

\caption{An example of the graph modification in Theorem  \ref{thmCO:iii}}
\label{fig:attachxgraph}
\end{figure}
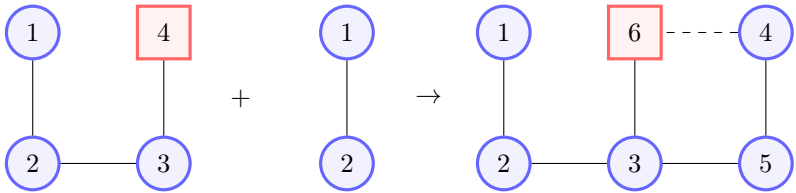

Note that the length of the shortest path between the first and last vertex will be the same for $G_A$ and $G_{\breve{A}}$. Thus, Conjecture \ref{con:CO} implies that this the contact order at $(-1,1)$ should be the same for both related polynomials. This is proved in the following result:

\begin{theorem} \label{thmCO:iii} Let $G_A$ be an $n$-vertex graph with $v_1$ and $v_{n-1}$ path connected, and $v_n$ only connected to $v_{n-1}$ and assume that contact order of $p_A$ at $(-1,1)$ is $K$. Then, with $\breve{A}$ defined in \eqref{eqn:A_breve}, the contact order of $p_{\breve{A}}$ at $(-1,1)$ is also $K$.
\end{theorem}

\begin{proof} To compare the contact orders, we will apply Remark \ref{rem:cor} to the functions
\[ g_A(z_1, z_2) : = f_A \left(-\tfrac{1}{z_1}, z_2 \right) = \frac{q_1}{q_2} \ \ \text{ and } \ \ g_{\breve{A}}(z_1, z_2) : = f_{\breve{A}}\left(-\tfrac{1}{z_1}, z_2 \right) = \frac{\breve{q}_1}{\breve{q}_2},  \]
where the pairs $q_1, q_2$ and $\breve{q}_1,\breve{q}_2$ have no common factor of $z_1$ or common factor containing $z_2$.
Let $\breve{z_Y}$ be the diagonal matrix (of appropriate size) with entries $-\frac{1}{z_1}, -\frac{1}{z_1}, \dots, -\frac{1}{z_1}, z_2$ along the diagonal. Then the structure of $A$  implies that \begin{align*}
g_A(z_1, z_2) = \frac{\det(\overline{A} - \breve{z_Y})}{\det(A - \breve{z_Y})} 
= \frac{z_2\det(\overline{B} + \frac{1}{z_1}I) + \det(\overline{\underline{B}} + \frac{1}{z_1}I)}{z_2\det(B + \frac{1}{z_1}I) + \det(\underline{B} + \frac{1}{z_1}I)},
\end{align*}
where the largest power of $\frac{1}{z_1}$ is exactly $n-1$ and so, we can take
\[
\begin{aligned} q_{1}(z_1,z_2) &= z_1^{n-1}\Big(z_2\det(\overline{B} + \tfrac{1}{z_1}I) + \det(\overline{\underline{B}} + \tfrac{1}{z_1}I)\Big), \\ 
q_{2}(z_1, z_2) &= z_1^{n-1}\Big(z_2\det(B + \tfrac{1}{z_1}I) + \det(\underline{B} + \tfrac{1}{z_1}I)\Big).\end{aligned}\]
As in the proofs of the previous two theorems, these $q_1, q_2$ share no common factor of $z_1$ or common factor involving $z_2$. From this, we can deduce
\[
\begin{aligned} r_1(z_1, z_2) &= z_1^{n-1}\Big(\det(\underline{B} + \tfrac{1}{z_1}I) -i \det(\overline{\underline{B}} + \tfrac{1}{z_1}I)\Big),\\
r_2(z_1, z_2) &= z_1^{n-1}\Big(\det(B + \tfrac{1}{z_1}I) - i \det(\overline{B} + \tfrac{1}{z_1}I) \Big).\end{aligned}\]
Thus, using the contact order assumption, we have
\begin{equation} \label{eqn:contact5}
\IM(r_1(x) \overline{r_2(x)}) = x^{2n-2} (\det(\underline{B} + \tfrac{1}{x}I)\det(\overline{B} + \tfrac{1}{x}I) - \det(B + \tfrac{1}{x}I)\det(\underline{\overline{B}} + \tfrac{1}{x}I)) 
\approx x^K.
\end{equation}

Now we connect this to the contact order of $p_{\breve{A}}$ or equivalently, to the order of vanishing of $\IM(\breve{r}_1(x) \overline{\breve{r}_2(x)})$ near $0$. To obtain a formula for $g_{\breve{A}}$, first recall that 
\[g_{\breve{A}}(z_1,z_2) = \frac{\det(\overline{\breve{A}} - \breve{z_Y})}{\det(\breve{A} - \breve{z_Y})}.\]

By applying the Laplace expansion repetitively on $\det(\breve{A} - \breve{z_Y})$ starting on bottom row of $\breve{A}$, and using the fact that the upper-right and the lower-left matrices in the partition of $\breve{A}$ have only a single row or column with $1$'s, one can show that there exist polynomials $s_1, s_2, s_3, s_4$ such that
%shift by one (n-m-1 is actually n-m)
\begin{align*}
\det(\breve{A} - \breve{z_Y}) =& z_2\Big(s_1(\tfrac{1}{z_1}) \det(B+\tfrac{1}{z_1} I) + s_2(\tfrac{1}{z_1}) \det(\underline{B}+\tfrac{1}{z_1} I)\Big) \\ &+ \Big(s_3(\tfrac{1}{z_1}) \det(B+\tfrac{1}{z_1} I) + s_4(\tfrac{1}{z_1}) \det(\underline{B}+\tfrac{1}{z_1} I)\Big).
\end{align*}
The highest power of $s_1$ is $m-n$ with a coefficient of $1$ by following the main diagonal.  The highest power of $s_4$ is also $m-n$ with a coefficient of $1$ by using the two $1$s at the last row and column, then following the main diagonal. The highest power for $s_2$ is $m-n-1$, using the fact that there is at least a pair of $1$s in the asterisks in $\breve{A}$.

Since the only nonzero values affected by moving from $\breve{A}$ to $\overline{\breve{A}}$ are in the change from $B$ to $\overline{B}$, we immediately obtain
\begin{align*}
\det(\overline{\breve{A}} - \breve{z_Y}) =& z_2\Big(s_1(\tfrac{1}{z_1}) \det(\overline{B}+\tfrac{1}{z_1} I) + s_2(\tfrac{1}{z_1}) \det(\overline{\underline{B}}+\tfrac{1}{z_1} I)\Big) \\ &+ \Big(s_3(\tfrac{1}{z_1}) \det(\overline{B}+\tfrac{1}{z_1} I) + s_4(\tfrac{1}{z_1}) \det(\overline{\underline{B}}+\tfrac{1}{z_1} I)\Big).
\end{align*}
Then, multiplying the formula for $g_{\breve{A}}$ by $z_1^{m-1}$ in the numerator and denominator will allow us to obtain the following formulas for $\breve{q}_1$ and $\breve{q}_2:$
\begin{align*}
\breve{q}_{1} &= z_1^{n-1}\Big(S_1(z_1)z_2 \det(\overline{B}+\tfrac{1}{z_1} I) + S_2(z_1)z_2 \det(\underline{\overline{B}}+\tfrac{1}{z_1} I) + S_3(z_1) \det(\overline{B}+\tfrac{1}{z_1} I) \\ & \hspace{.5in} + S_4(z_1) \det(\underline{\overline{B}}+\tfrac{1}{z_1} I)\Big)\\
\breve{q}_{2} &= z_1^{n-1}\Big(S_1(z_1)z_2 \det(B+\tfrac{1}{z_1} I) + S_2(z_1)z_2 \det(\underline{B}+\tfrac{1}{z_1} I) + S_3(z_1) \det(B+\tfrac{1}{z_1} I)  \\ & \hspace{.5in} + S_4(z_1) \det(\underline{B}+\tfrac{1}{z_1} I)\Big),
\end{align*}
where $S_k(z_1) = z_1^{m-n}s_k(\frac{1}{z_1})$ for $k=1, \dots, 4.$ 

By manipulating the equations $\breve{q}_1$ and $\breve{q}_2$, we obtain
\[ \begin{aligned}\breve{r}_1 &= z_1^{n-1}\Big(S_3(z_1) \det(B+\tfrac{1}{z_1} I) + S_4(z_1) \det(\underline{B}+\tfrac{1}{z_1} I) -iS_3(z_1) \det(\overline{B}+\tfrac{1}{z_1} I) \\ & \hspace{.5in} -iS_4(z_1) \det(\underline{\overline{B}}+\tfrac{1}{z_1} I) \Big) \\
\breve{r}_2 &= z_1^{n-1}\Big(S_1(z_1)\det(B+\tfrac{1}{z_1} I) + S_2(z_1)\det(\underline{B}+\tfrac{1}{z_1} I)-iS_1(z_1) \det(\overline{B}+\tfrac{1}{z_1} I) \\ & \hspace{.5in}  -iS_2(z_1) \det(\underline{\overline{B}}+\tfrac{1}{z_1} I)\Big). \end{aligned}\]
By substituting and canceling terms, we have
\begin{align*}
 \IM(\breve{r}_1(x) &\overline{\breve{r}_2(x)})  \\
=& x^{2n-2} \Big(\det(\underline{B} + \tfrac{1}{x}I)\det(\overline{B} + \tfrac{1}{x}I) - \det(B + \tfrac{1}{x}I)\det(\underline{\overline{B}} + \tfrac{1}{x}I)\Big) \\
& \hspace{1in} \times \Big(S_1(x)S_4(x)-S_2(x)S_3(x)\Big) \\
\approx & x^K\Big(S_1(x)S_4(x)-S_2(x)S_3(x)\Big),
\end{align*}
where we used \eqref{eqn:contact5}.
By the earlier statements about the highest powers of the $s_k$, $S_1(0), S_4(0) \ne 0$ and $S_2(0)=0$, so the lowest order term in $S_1(x)S_4(x)-S_2(x)S_3(x)$ must be a nonzero constant term. Thus the contact order of $p_{\breve{A}}$ at $(-1,1)$ is also $K$, as needed.
\end{proof}

\section{Open Questions} \label{sec:open}

This paper initiates the study of stable polynomials constructed via general graphs and a specific type of coloring matrix $Y$. The line of inquiry is still quite new and many questions remain open. Here are three possible directions of inquiry: \\

\begin{itemize}

\item[1.] \emph{What can one say about additional boundary zeros of the stable polynomials $p_A^t$?} A number of special cases were studied in \cite{H23}, but we currently know of no general way to determine (from the graph's properties) whether the polynomial will have extra boundary zeros.  \\

\item [2.] \emph{What happens when one changes the coloring matrix $Y$ or vector $\alpha$?} For our particular choice of $Y$, the existence of at least one boundary zero was fully determined by whether the first and last vertices of the graph were path connected. It would be very interesting if there was a more  general characterization of the existence of boundary zeros that held for stable polynomials constructed from general triples ($A$, $Y$, $\alpha$). In particular, changing the matrix $Y$ would allow for much more complicated (higher degree) polynomials. \\

\item [3.] \emph{What additional information do these types of stable polynomials encode about the associated colored graphs?} This current paper was mostly concerned with using  graphs to obtain polynomials with particular properties. However, one could imagine going in the other direction and using a polynomial to obtain information about an underlying graph.  For example, can one use these  polynomials (or related objects) to count the number of paths with certain properties within the colored graph? \\ 

\end{itemize}

\end{document}